\documentclass[11pt]{article} 
\edef\savecatcodeat{\the\catcode`@}
\catcode`\@=11

\def\@IFNEXTCHAR#1#2#3{\let\@tempe #1\def\@tempa{#2}\def\@tempb{#3}\futurelet
    \@tempc\@IFNCH}
\def\@IFNCH{\ifx \@tempc \@sptoken \let\@tempd\@xifnch
      \else \ifx \@tempc \@tempe\let\@tempd\@tempa\else\let\@tempd\@tempb\fi
      \fi \@tempd}

\def\tb@ifSpecChars#1#2{#1}
\def\tb@ifNoSpecChars#1#2{#2}

\newcount\tb@rpN

\def\tableau{%
  \bgroup
  \@IFNEXTCHAR[{\tb@tableauC}{\tb@tableauC[]}}     

\def\tb@tableauC[#1]{\hbox\bgroup%
    \let\\=\cr
    \def\bl{\global\let\tbcellF\tb@cellNF}%
    \def\tf{\global\let\tbcellF\tb@cellH}
    \dimen2=\ht\strutbox \advance\dimen2 by\dp\strutbox%
    \ifx\baselinestretch\undefined\relax%
    \else%
       \dimen0=100sp \dimen0=\baselinestretch\dimen0%
       \dimen2=100\dimen2 \divide\dimen2 by\dimen0%
    \fi%
    \let\tpos\tb@vcenter
    \tb@initYoung
    \tb@options#1\eoo
    \let\arrow\tb@arrow%
    \dimen0=\Tscale\dimen2%
    \dimen1=\dimen0 \advance\dimen1 by \tb@fframe%
    \lineskip=0pt\baselineskip=0pt
    \def\tb@nothing{}%
    \def\endcellno{$\rss\egroup\bss\egroup}
    \def\endcell{\endcellno\kern-\dimen0}
    \def\begincell{\vbox to\dimen0\bgroup\vss\hbox to\dimen0\bgroup\hss$}%
    \let\overlay\tb@overlay%
    \let\fl\tb@fl%
    \let\fr\tb@fr%
    \let\lss\hss\let\rss\hss\let\tss\vss\let\bss\vss
    \def\mkcell##1{
        \let\tbcellF\tb@cellD
        \def\tb@cellarg{##1}
        \ifx\tb@cellarg\tb@nothing\let\tb@cellarg\tb@cellE\fi%
	        \begincell\tb@cellarg\endcellno
	        \tbcellF
    }%
    \let\savecellF\tbcellF
    \tb@tableauD%
}%

\let\tb@savetableauD\tableauD
{
\gdef\tableauD#1{%
  \tpos{\tabskip=0pt\halign{&\mkcell{##}\cr#1\crcr}}%
  \global\let\tbcellF\savecellF
  \egroup
  \egroup}
}
\let\tb@tableauD\tableauD
\let\tableauD\tb@savetableauD
\let\tb@savetableauD\undefined

\def\tb@options#1{\ifx#1\eoo\relax\else\tb@option#1\expandafter\tb@options\fi}

\def\tb@option#1{%
  \if#1t\let\tpos\tb@vtop\fi
  \if#1c\let\tpos\tb@vcenter\fi
  \if#1b\let\tpos\vbox\fi
  \if#1F\tb@initFerrers\fi
  \if#1Y\tb@initYoung\fi
  \if#1E\tb@initEmpty\fi
  \if#1s\tb@initSmall\fi
  \if#1m\tb@initMedium\fi
  \if#1l\tb@initLarge\fi
  \if#1p\tb@initPartition\fi
  \if#1a\tb@initArrow\fi
}

\def\tb@vcenter#1{\ifmmode\vcenter{#1}\else$\vcenter{#1}$\fi}

\def\tb@vtop#1{\hbox{\raise\ht\strutbox\hbox{\lower\dimen0\vtop{#1}}}}

\def\tb@initPartition{\def\Tscale{.3}}
\def\tb@initSmall{\def\Tscale{1}}
\def\tb@initMedium{\def\Tscale{2}}
\def\tb@initLarge{\def\Tscale{3}}

\def\tb@initArrow{\dimen2=1.25em}

\def\tb@initYoung{%
  \def\tb@cellE{}
  \let\tb@cellD\tb@cellN
  \def\sk{\global\let\tbcellF\tb@cellNF}}
\def\tb@initFerrers{%
  \def\tb@cellE{\bullet}
  \let\tb@cellD\tb@cellNF
  \def\sk{\bullet}}
\def\tb@initEmpty{%
  \def\tb@cellE{}
  \let\tb@cellD\tb@cellNF
  \def\sk{\global\let\tbcellF\tb@cellNF}}

\tb@initMedium

\def\tb@sframe#1{%
  \vbox to0pt{
    \vss
    \hbox to0pt{%
      \hss
      \vbox to\dimen1{
        \hrule depth #1 height0pt
        \vss
        \hbox to\dimen1{
          \vrule width #1 height\dimen1
          \hss
          \vrule width #1
          }%
        \vss
        \hrule height #1 depth 0in
        }%
      \kern-\tb@hframe
      }%
    \kern-\tb@hframe}}

\def\tb@hframe{.2pt}\def\tb@fframe{.4pt}\def\tb@bframe{2pt}
\def\tb@cellH{\tb@sframe{\tb@bframe}}       
\def\tb@cellNF{}                            
\def\tb@cellN{\tb@sframe{\tb@fframe}}       
\let\tbcellF\tb@cellN                       

\def\tb@Fsframe{%
  \vbox to0pt{
    \vss
    \hbox to0pt{%
      \hss
      \vbox to\dimen1{
        \fr@iftop{\hrule depth \fr@width height0pt}{\vskip \fr@width}
        \vss
        \hbox to\dimen1{
	  \fr@ifleft{\vrule width \fr@width height\dimen1}{\hskip \fr@width}
          \hss
          \fr@ifright{\vrule width \fr@width height\dimen1}{\hskip \fr@width}
          }%
        \vss
        \fr@ifbottom{\hrule height \fr@width depth 0in}{\vskip\fr@width}
        }%
      \kern-\tb@hframe
      }%
    \kern-\tb@hframe}}

\def\tb@fr{\@IFNEXTCHAR[{\tb@fra}{\global\let\tbcellF\tb@cellN}}
\def\tb@fra[#1]{%
	\global\let\fr@iftop\tb@IFNO
	\global\let\fr@ifbottom\tb@IFNO%
	\global\let\fr@ifleft\tb@IFNO%
	\global\let\fr@ifright\tb@IFNO%
	\global\let\fr@width\tb@fframe%
	\global\let\tbcellF\tb@Fsframe%
	\froptions#1\eoo
}
\def\froptions#1{\ifx#1\eoo\relax\else\froption#1\expandafter\froptions\fi}
\def\froption#1{
	\if#1t\global\let\fr@iftop\tb@IFYES\fi
	\if#1b\global\let\fr@ifbottom\tb@IFYES\fi
	\if#1l\global\let\fr@ifleft\tb@IFYES\fi
	\if#1r\global\let\fr@ifright\tb@IFYES\fi
	\if#1w\global\let\fr@width\tb@bframe\fi
}
\def\tb@IFYES#1#2{#1}
\def\tb@IFNO#1#2{#2}

\def\tb@rpad{1pt}
\def\tb@lpad{1pt}
\def\tb@tpad{1.8pt}
\def\tb@bpad{1.8pt}

\def\tb@overlay{\endcell\@IFNEXTCHAR[{\tb@overlaya}{\begincell}}
\def\tb@overlaya[#1]{\vbox to\dimen0\bgroup%
  \tb@overlayoptions#1\eoo%
  \tss\hbox to\dimen0\bgroup\lss$}
\def\tb@overlayoptions#1{\ifx#1\eoo\relax\else\tb@overlayoption#1\expandafter\tb@overlayoptions\fi}

\def\tb@overlayoption#1{
  \if#1t\def\tss{\vskip\tb@tpad}\let\bss\vss\fi
  \if#1c\let\tss\vss\let\bss\vss\fi
  \if#1b\def\bss{\vskip\tb@bpad}\let\tss\vss\fi
  \if#1l\def\lss{\hskip\tb@lpad}\let\rss\hss\fi
  \if#1m\let\lss\hss\let\rss\hss\fi
  \if#1r\def\rss{\hskip\tb@rpad}\let\lss\hss\fi
}

\def\tb@fl{\endcell\begincell\vrule depth 0pt width \dimen0 height \dimen0 \endcell\begincell}

\def\tbgobble#1{}
\def\Pscale{1}

\def\skewptn{%
  \@IFNEXTCHAR[{\tb@ptnC}{\tb@ptnC[]}}     

\def\tb@ptnC[#1](#2){%
	{%
    \let\Tscale\Pscale
    \let\\=\cr
   \def\tb@initYoung{%
	\def\tb@cell{\hskip\dimen0\tb@cellN}%
	\def\tb@kernA{\kern.5\dimen0}%
	\def\tb@kernB{\kern-.5\dimen0}%
   }%
   \def\tb@initFerrers{%
	\def\tb@cell{\hbox to\dimen0{\hss$\bullet$\hss}}%
	\def\tb@kernA{}%
	\def\tb@kernB{}%
   }%
    \dimen2=\ht\strutbox \advance\dimen2 by\dp\strutbox%
    \ifx\baselinestretch\undefined\relax%
    \else%
       \dimen0=100sp \dimen0=\baselinestretch\dimen0%
       \dimen2=100\dimen2 \divide\dimen2 by\dimen0%
    \fi%
    \let\tpos\tb@vcenter
    \tb@initYoung
    \tb@options#1\eoo
    \dimen0=\Tscale\dimen2%
    \dimen1=\dimen0 \advance\dimen1 by \tb@fframe%
    \lineskip=0pt\baselineskip=0pt
    \tpos{\skewptnDnewline#2|)}%
	}%
}%

\def\skewptnDnewline#1|{\vbox to\dimen0\bgroup\vss\tb@kernA\hbox\bgroup\skewptnEon#1,|}
\def\skewptnDendline|{\egroup\tb@kernB\vss\egroup\@IFNEXTCHAR{)}{\tbgobble}{\skewptnDnewline}}
\def\skewptnEon#1,{%
	\tb@rpN=#1%
	\ifnum#1>0
	        \loop%
		\tb@cell%
	        \ifnum\tb@rpN>1\advance\tb@rpN by-1%
        	\repeat%
	\fi%
	\@IFNEXTCHAR{|}{\skewptnDendline}{\skewptnEoff}}
\def\skewptnEoff#1,{\hskip #1\dimen0%
	\@IFNEXTCHAR{|}{\skewptnDendline}{\skewptnEon}}

\catcode`\@=\savecatcodeat
\let\savecatcodeat\undefined

\usepackage{latexsym, amssymb,amsmath,mathtools,amsthm,amscd,xy,geometry,caption}
\usepackage{graphicx}
\usepackage[bookmarks, bookmarksdepth=2, colorlinks=true, linkcolor=blue, citecolor=blue, urlcolor=blue]{hyperref}
\usepackage{pgfplots}

\def\Q{{\mathbb Q}}
\def\P{{\mathbb P}}

\def\RR{{\mathbb R}}
\def\C{{\mathbb C}}
\def\OO{{\mathcal O}}

\def\kr{{\mathbb C}}
\newcommand{\mto}[1]{\stackrel{#1}\longrightarrow}
\newcommand{\lmto}[1]{\stackrel{#1}\longleftarrow}
\def\cF{{\mathcal F}}
\def\cE{{\mathcal E}}

\def\cQ{{\mathcal Q}}
\def\cO{{\mathcal O}}
\def\cK{{\mathcal K}}
\def\ZZ{{\mathbb Z}}
\def\PP{{\mathbb P}}
\def\VV{{\mathbb V}}
\newcommand{\Sym}{\text{Sym}}
\def\te{\otimes}
\def\bfone{{\mathbf 1}}
\def\dt{\bullet}
\def\Ext{{\text{Ext}}}
\def\iso{\cong}
\newcommand{\vmto}[1]{\stackrel{#1}\longleftarrow}
\def\Hom{\text{Hom}}
\def\oval{{\underline{ \alpha}}}
\def\ovphi{{\underline{\phi}}}
\def\ovalt{{\underline{\alpha}}^{T}}

\newcommand{\arxiv}[1]{\href{http://arxiv.org/abs/#1}{{\tt arXiv:#1}}}

\newtheorem{theorem}{Theorem}[section]
\newtheorem{lemma}[theorem]{Lemma} 
\newtheorem{prop}[theorem]{Proposition} 
\newtheorem{definition}[theorem]{Definition}

\newtheorem{remark}[theorem]{Remark}
\let\olddefinition\remark
\renewcommand{\remark}{\olddefinition\normalfont}

\newtheorem{example}[theorem]{Example}  
 
\newtheorem{fact}[theorem]{Fact}
\newtheorem{question}[theorem]{Practical Question}
\newtheorem{empirical}[theorem]{Empirical Result}

 \newcommand\blfootnote[1]{%
  \begingroup
  \renewcommand\thefootnote{}\footnote{#1}%
  \addtocounter{footnote}{-1}%
  \endgroup
}

\numberwithin{equation}{section}                       
\begin{document}
\title{The Chow Form of the Essential Variety in Computer Vision}
\author{Gunnar Fl\o ystad, Joe Kileel, Giorgio Ottaviani}
\date{}

\maketitle
\begin{abstract} \noindent
The Chow form of the 
essential variety in computer
vision is calculated.   
Our 
derivation uses secant varieties, 
Ulrich sheaves and representation theory. 
Numerical experiments show
that our formula can detect
noisy point 
correspondences
between two images.
\end{abstract}

\blfootnote{\textit{2010 Mathematics Subject Classification.} 14M12, 14C05, 14Q15, 13D02, 13C14, 68T45.}
\blfootnote{\textit{Key words and phrases.} Chow form, Ulrich sheaf, Pieri resolutions, calibrated cameras, essential variety.} 

\section{Introduction}\label{sec:intro}
The {\it essential variety} ${\mathcal E}$ is the variety of $3\times 3$ real matrices with two equal singular values, and the third one equal to zero ($\sigma_1=\sigma_2$, $\sigma_3=0$).
It was introduced in the setting of computer vision; see \cite[\S 9.6]{HZ}. Its elements, the so-called {\it essential matrices}, have the form $TR$, where $T$
is real skew-symmetric and $R$ is real orthogonal. 
The essential variety is a cone of codimension $3$ and degree $10$ in the space of $3 \times 3$-matrices, defined by homogeneous cubic equations, that we recall in (\ref{eq:ess}).
The complex solutions of these cubic equations define the complexification 
${\mathcal E}_\C$ of the essential variety. While the real essential variety is smooth, its complexification has a singular locus that we describe precisely in \S \ref{sec:determinantal}.

The {\it Chow form} of a codimension $c$ projective variety $X\subset\P^n$
is the equation $\textup{Ch}(X)$ of the divisor in the Grassmannian $\textup{Gr}(\P^{c-1},\P^n)$
given by those linear subspaces of dimension $c-1$ which meet $X$.
It is a basic and classical tool that allows one to recover much geometric information about $X$; 
for its main properties we refer to \cite[\S 4]{GKZ}.
In \cite[\S 4]{ALST}, the problem of computing the Chow form of the essential variety 
was posed, while the analogous problem for the {\it fundamental variety} was solved, 
another important variety in computer vision. 

The main goal of this paper is to explicitly find the Chow form of the essential variety.
This provides an important tool for the problem of detecting if a set of image point correspondences 
$\{(x^{(i)},y^{(i)}) \in \RR^2\times\RR^2 \, | \, i=1,\ldots, m\}$ comes from $m$ world points
in $\RR^3$ and two calibrated cameras.  It furnishes an exact solution for $m=6$ and it behaves well 
given noisy input, as we will see in \S \ref{sec:Chow}.  Mathematically, we can consider the system of equations:
\begin{equation} \label{3d}
\begin{cases}
A\widetilde{X^{(i)}} \equiv \widetilde{x^{(i)}} \\
B\widetilde{X^{(i)}} \equiv \widetilde{\, y^{(i)}}. 
\end{cases}
\end{equation}

\noindent Here $\widetilde{x^{(i)}} = (x^{(i)}_{1} \colon x^{(i)}_{2} \colon 1)^{T} \in \P^2$
and $\widetilde{\, y^{(i)}} = (y^{(i)}_{1} \colon y^{(i)}_{2} \colon 1)^{T} \in \P^2$
are the given image points.
The unknowns are two $3 \times 4$ matrices $A, B$
with rotations in their left $3 \times 3$ blocks 
and $m = 6$ points $\widetilde{X^{(i)}} \in \P^{3}$. 
These represent calibrated cameras 
and world points, respectively.  
A calibrated camera has normalized image coordinates, as explained in \cite[\S 8.5]{HZ}.
Here $\equiv$ denotes equality up to nonzero scale.  
From our calculation of $\textup{Ch}(\mathcal{E}_{\C})$, we deduce:

\begin{theorem}\label{mainThm}
There exists an explicit $20 \times 20$ skew-symmetric matrix $\mathcal{M}(x,y)$ of degree $\le (6,6)$ polynomials 
over $\ZZ$ in the coordinates of $(x^{(i)}, y^{(i)})$ with the following properties. If 
\textup{(\ref{3d})} admits a complex solution then $\mathcal{M}(x^{(i)}, y^{(i)})$ is rank-deficient. Conversely,
the variety of point correspondences $(x^{(i)}, y^{(i)})$ such that $\mathcal{M}(x^{(i)}, y^{(i)})$ is rank-deficient
contains a dense subset for which \textup{(\ref{3d})} admits a complex solution.
\end{theorem}

\noindent In fact, we will produce two such matrices.  Both of them, along with related formulas we derive, 
are available in ancillary files accompanying the \texttt{arXiv} version of this paper, and we have posted them at
\url{http://math.berkeley.edu/~jkileel/ChowFormulas.html}.

Our construction of the Chow form
uses the technique of \textit{Ulrich sheaves} introduced by Eisenbud and Schreyer in \cite{ESW}.
We construct rank $2$ Ulrich sheaves on the essential variety $\mathcal{E}_{\C}$.
For an analogous construction of the Chow form of $K3$ surfaces,
see \cite{AFO}.

From the point of view of computer vision, this paper contributes a complete
 characterization for an `almost-minimal' problem. 
Here the motivation is {\it 3D reconstruction}.  Given multiple images 
of a world scene, taken by cameras in an unknown configuration, 
we want to estimate the camera configuration and a 3D model
of the world scene.  Algorithms for this are complex, and
successful.  See \cite{ASSSS} for a reconstruction from 150,000 images.

By contrast, the system (\ref{3d}) encodes a tiny reconstruction problem.
Suppose we are given six point correspondences in two calibrated
pictures (the right-hand sides in (\ref{3d})).
We wish to reconstruct both the two cameras and the six world points 
(the left-hand sides in (\ref{3d})).
If an exact solution exists then it is typically unique, modulo the natural
symmetries. However, an exact solution does not always exist. In order for this to happen, a giant
polynomial of degree 120 in the 24 variables on the right-hand sides has to vanish.
Theorem \ref{mainThm} gives an explicit matrix formula for that polynomial.

The link between minimal or almost-minimal
reconstructions and large-scale reconstructions 
is surprisingly strong.  Algorithms for the latter use the former reconstructions
repeatedly as core subroutines.   In particular, 
solving the system (\ref{3d}) given $m=5$ point pairs, instead of $m=6$, is a
subroutine in \cite{ASSSS}.  This solver is optimized in \cite{Nis}.
It is used to generate hypotheses inside \textit{random sampling consensus} 
(RANSAC) \cite{FB}
schemes for robust reconstruction from pairs of
calibrated images.  See \cite{HZ} for more vision background.

This paper is organized as follows.  In \S \ref{sec:determinantal}, we prove 
that ${\mathcal E}_\C$ is a hyperplane section of the variety $PX^s_{4,2}$ of $4\times 4$ symmetric matrices of rank $\le 2$. 
This implies a determinantal description of ${\mathcal E}_\C$; see Proposition \ref{prop:sm}.
A side result of the construction is that ${\mathcal E}_\C$ is the secant variety of its singular locus, which corresponds to pairs of isotropic vectors in $\C^3$.

In \S \ref{sec:Ulrich}, we construct two Ulrich sheaves on the variety of $4\times 4$ symmetric matrices of rank $\le 2$. 
One of the constructions we propose is new to our knowledge.
Both sheaves are GL(4)-equivariant, and they admit ``Pieri resolutions" in the sense of \cite{SW}.
We carefully analyze the resolutions using representation theory, and in particular show that
their middle differentials may be represented by symmetric matrices; 
see Propositions \ref{symM} and \ref{symN}.

In \S \ref{sec:Chow}, we combine the results of the previous sections and we construct the Chow form of the essential variety. 
The construction from \cite{ESW} starts with our rank $2$ Ulrich sheaves and allows to define two $20 \times 20$
matrices in the Pl\"ucker coordinates of $\textup{Gr}(\P^2,\P^8)$ each of which drops rank exactly
when the corresponding subspace $\P^2$ meets the essential variety $\mathcal{E}_{\C}$.
It requires some technical effort to put these matrices in skew-symmetric form,
and here our analysis from \S \ref{sec:Ulrich} pays off.  
We conclude the paper with numerical experiments demonstrating the robustness to noise 
that our matrix formulas in Theorem \ref{mainThm} enjoy. 

\bigskip

\noindent {\bf Acknowledgements.}\ The authors are grateful to
Bernd Sturmfels for his interest and encouragement.  They thank
Anton Fonarev pointing out the connection to Littlewood complexes 
in Remark \ref{rem:anton}.  J.K. and G.O. 
are grateful to Frank Schreyer for very useful conversations.  
G.F. and J.K. thank Steven Sam for help with \texttt{PieriMaps}
and for suggesting references to show that the middle maps $\phi$ are symmetric.
J.K. thanks Justin Chen for valuable 
comments.  G.O. is a member of GNSAGA-INDAM.

\section{The essential variety is a determinantal variety}\label{sec:determinantal}
\subsection{Intrinsic description}
Let ${\mathcal E}\subset{\mathbb R}^{3\times 3}$ be the essential variety defined by the conditions: 
$${\mathcal E}:=\{M\in {\mathbb R}^{3\times 3} \, | \, \sigma_1(M)=\sigma_2(M), \,\sigma_3(M)=0\}.$$ 

The polynomial equations of ${\mathcal E}$ are (see \cite[\S 4]{FM}) as follows:
\begin{equation}\label{eq:ess}
{\mathcal E}=\{M\in {\mathbb R}^{3\times 3}\,|\, \det(M)=0,\, 2(MM^{T})M- \text{tr}\left(MM^{T}\right)M=0\}.\end{equation}

These 10 cubics minimally generate the \textit{real radical ideal} \cite[p.85]{BCR}
of the essential variety $\mathcal{E}$, and that ideal is prime.
Indeed, the real radical property follows from our Proposition \ref{prop:singessential}(i)
and \cite[Theorem 12.6.1]{Mar}.
We denote by ${\mathcal E}_\C$ the projective variety in $\P_\C^8$ given by the complex solutions of  (\ref{eq:ess}).  
The essential variety ${\mathcal E}_\C$ has codimension $3$ and degree $10$.
In this section, we will prove that it is isomorphic to a hyperplane section of the variety $PX^{s}_{4,2}$ of complex symmetric $4\times 4$ matrices of rank $\le 2$. 
The first step towards this is Proposition \ref{prop:singessential} below,
and that relies on the group symmetries of $\mathcal{E}_{\C}$, which
we now explain.

Consider ${\RR}^{3}$ with the standard inner product $Q$, and the corresponding action of $\text{SO}(3,\RR)$ on ${\mathbb R}^{3}$. 
Complexify ${\RR}^{3}$ and consider ${\C}^{3}$ with the action of $\text{SO}(3,\C)$,
which has universal cover $\text{SL}(2,\C)$. It is technically simpler to work with the action of $\text{SL}(2,\C)$. Denoting by $U$ the irreducible $2$-dimensional representation of $\text{SL}(2,\C)$, we have the equivariant isomorphism $\C^3\cong S_{2}U$.  Writing $Q$ also for the complexification of the Euclidean product,
the projective space $\PP(S_{2}U)$ divides into two $\text{SL}(2,\C)$-orbits,
namely the isotropic quadric with equation $Q(u)=0$ and its complement.
Let $V$ be another complex vector space of dimension $2$.
The essential variety $\mathcal{E}_{\C}$ is embedded into the projective space of
$3 \times 3$-matrices $\PP(S_{2}U\otimes S_{2}V)$.
Since the singular value conditions defining $\mathcal{E}$ are 
$\text{SO}(3,\RR) \times \text{SO}(3,\RR)$-invariant, it follows that
$\mathcal{E}_{\C}$ is $\text{SL}(U) \times \text{SL}(V)$-invariant using \cite[Theorem 2.2]{DLOT}.

The following is a new geometric description of the essential variety.  
From the computer vision application, we start with the set of 
real points $\mathcal{E}$.  However, below we see that the surface
$\textup{Sing}(\mathcal{E}_{\C})$ inside $\mathcal{E}_{\C}$, which has
no real points, `determines' the algebraic geometry. Part (i) of Prop. \ref{prop:singessential}
is proved also in \cite[Prop. 5.9]{Maybank}.

\begin{prop}\label{prop:singessential}
(i) The singular locus of ${\mathcal E}_\C$ is the projective surface given by:
$$\textup{Sing}({\mathcal E}_\C)=\left\{a\cdot b^T\in \PP({\C}^{3\times 3}) \,| \, Q(a)=Q(b)=0\right\}.$$

(ii) The second secant variety of ${\textup{Sing}}({\mathcal E}_\C)$ equals ${\mathcal E}_\C$.
\end{prop}

\begin{proof} Denote by $S$ the variety $\left\{a\cdot b^{T}\in \PP({\C}^{3\times 3})\,| \,Q(a)=Q(b)=0\right\}$, 
and let $\widehat{S}$ be the affine cone over it.
The line secant variety $\sigma_2(\widehat{S})$ consists of elements of the form
$M=a_1b_1^T+a_2b_2^T \in {\C}^{3\times 3}$ such that $Q(a_i)=a_i^Ta_i=Q(b_i)=b_i^Tb_i=0$ for $i=1,2$.  We compute that
$MM^T \, = \, a_1b_1^Tb_2a_2^T+a_2b_2^Tb_1a_1^T$ so that $\text{tr}(MM^T) \, = \, 2(b_1^Tb_2)(a_1^Ta_2)$.
Moreover $MM^TM \, = \, a_1b_1^Tb_2a_2^Ta_1b_1^T+a_2b_2^Tb_1a_1^Ta_2b_2^T \, = \, (b_1^Tb_2)(a_1^Ta_2)M$. 
Hence the equations (\ref{eq:ess}) of ${\mathcal E}_\C$ are satisfied by $M$. This proves that 
$\sigma_2(S)\subset {\mathcal E}_\C$.  Since
$\sigma_2(S)$ and ${\mathcal E}_\C$ are both of codimension $3$ and 
${\mathcal E}_\C$ is irreducible, the equality $\sigma_2(S) = {\mathcal E}_\C$ follows. It remains to prove (i). 
Denote by $[a_i]$  the line generated by
$a_i$. 
Every element $a_1b_1^T+a_2b_2^T$ with $[a_1]\neq[a_2]$, $[b_1]\neq[b_2]$ and $Q(a_i) = Q(b_i)=0$ for $i=1,2$ can be taken by $\textup{SL}(U)\times \textup{SL}(V)$
to a scalar multiple of any other element of the same form. 
This is the open orbit of the action of $\textup{SL}(U)\times \textup{SL}(V)$
on ${\mathcal E}_\C$.
The remaining orbits are the following:
\begin{enumerate}
\item{} the surface $S$, with set-theoretic equations $ MM^T = M^TM = 0$.
\item{} $T_1\setminus S$, where $T_1=\left\{a\cdot b^{T}\in \PP({\C}^{3\times 3})\,| \,Q(a)=0\right\}$ is a threefold, with set-theoretic equations $ M^TM = 0$.
\item{} $T_2\setminus S$, where $T_2=\left\{a\cdot b^{T}\in \PP({\C}^{3\times 3})\,| \,Q(b)=0\right\}$ is a threefold, with set-theoretic equations $ MM^T  = 0$.
\item{} $\textup{Tan}(S) \setminus(T_1\cup T_2)$, where the \textit{tangential variety} $\textup{Tan}(S)$
is the fourfold union of all tangent spaces to $S$, 
with set-theoretic equations $\textup{tr}(MM^T)=0, MM^TM=0$.
\end{enumerate}

One can compute explicitly that the Jacobian matrix of
${\mathcal E}_\C$ at $\begin{pmatrix}
1&0&0\\
\sqrt{-1}&0&0\\
0&0&0\end{pmatrix}\in T_1\setminus S$ has rank $3$.  The following 
code in \texttt{Macaulay2} \cite{GS} does that computation:

\begin{verbatim}
R = QQ[m_(1,1)..m_(3,3)]
M = transpose(genericMatrix(R,3,3))
I = ideal(det(M))+minors(1,2*M*transpose(M)*M - trace(M*transpose(M))*M)
Jac = transpose jacobian I
S = QQ[q]/(1+q^2)
specializedJac = (map(S,R,{1,0,0,q,0,0,0,0,0}))(Jac)
minors(3,specializedJac)
\end{verbatim}

Hence the points
in $T_1\setminus S$ are smooth points of ${\mathcal E}_\C$. 
By symmetry, also the points in $T_2\setminus S$ are smooth.
By semicontinuity, the points in $\textup{Tan}(S)\setminus(T_1\cup T_2)$ are smooth.
Since points in $S$ are singular for the secant variety $\sigma_2(S)$, this finishes the proof of (i).
\end{proof}

\begin{remark} From the study of tensor decomposition,
the parametric description in Proposition \textup{\ref{prop:singessential}} 
is identifiable.  That shows that real essential matrices have the form
$a^Tb+\overline{a}^T\overline{b}$ with $a, b\in\C^3$ and $Q(a)=Q(b)=0$.
This may be written in the alternative form
$(u^2)^Tv^2+(\overline{u}^2)^T\overline{v}^2\in S_{2}(U)\otimes S_{2}(V) $ with 
$u\in U$, $v\in V$.
This may help in computing real essential matrices.
Note that the four non-open orbits listed in the proof of
Proposition \ref{prop:singessential} are contained in the isotropic quadric
 $\textup{tr}(MM^T)=0$, hence they have no real points.
\end{remark}

\begin{remark}The surface $\textup{Sing}({\mathcal E}_\C)$ is more familiar with the embedding by $\OO(1,1)$, when it is the smooth quadric surface, doubly ruled by lines. In the embedding by $\OO(2,2)$, the two
rulings are given by conics.  These observations suggests expressing ${\mathcal E}_\C$ as a 
determinantal variety, as we do next in Proposition \textup{\ref{prop:eddegree}}.
Indeed, note that the smooth quadric surface embedded by $\OO(2,2)$ 
is isomorphic to a linear section of the second Veronese embedding of $\P^3$,
which is the variety of $4\times 4$ symmetric matrices of rank $1$. 
\end{remark}

\begin{prop}\label{prop:eddegree} The essential variety ${\mathcal E}_\C$ is isomorphic to a hyperplane section
of the variety of rank $\le 2$ elements in $\P(S_{2}(U\otimes V))$.  Concretely, that ambient space 
identifies with the projective variety of $4\times 4$ symmetric matrices of rank $\le 2$, denoted by $PX^s_{4,2}$,
and the section consists of traceless $4 \times 4$ symmetric matrices of rank $\le 2$.
\end{prop}
\begin{proof}
The embedding of $\P(U)\times\P(V)$
in $\P(S_{2}(U)\otimes S_{2}(V))$ is given by $(u, v)\mapsto u^2\otimes v^2$. 
Recall that Cauchy's formula states $S_{2}(U \otimes V)=
\left(S_{2}(U)\otimes S_{2}(V)\right)\oplus \left(\wedge^{2}U\otimes\wedge^{2}V\right)$, where 
$\dim(U\otimes V) = 4$.  Hence, $\P(S_{2}(U)\otimes S_{2}(V))$
is equivariantly embedded as a codimension one subspace in $\P(S_{2}(U\otimes V))$. 
The image is the subspace of traceless elements, and this map sends
$u^2\otimes v^2 \mapsto (u \otimes v)^2$.
By Proposition \ref{prop:singessential}, 
we have shown that
$\textup{Sing}(\mathcal{E}_{\C})$ 
embeds into a
hyperplane section of the variety of rank 1 elements
in $\P(S_{2}(U\otimes V))$.
So, $\mathcal{E}_{\C} = \sigma_{2}(\text{Sing}(\mathcal{E}_{\C}))$
embeds into that hyperplane section of
the variety of rank $\leq 2$ elements.
Comparing dimensions and degrees,
the result follows.
\end{proof}

\begin{remark} In light of the description in Proposition \textup{\ref{prop:eddegree}}, 
it follows by Example \textup{3.2} and Corollary \textup{6.4} of \textup{\cite{DHOST}}
that the Euclidean distance degree is
$\textup{EDdegree}(\mathcal{E}_{\C})=6$. This result has been proved
also in \textup{\cite{DLT}}, where the computation of $\textup{EDdegree}$
was performed in the more general setting of orthogonally invariant varieties.
This quantity measures the algebraic complexity of finding the nearest 
point on $\mathcal{E}$ to a given noisy data point in $\RR^{3 \times 3}$.
\end{remark}

\subsection{Coordinate description}\label{subsec:coord}
We now make the determinantal description of 
$\mathcal{E}_{\C}$ in Proposition \ref{prop:eddegree} 
explicit in coordinates.
For this, denote $a=(a_1, a_2, a_3)^T \in \C^3$. We have $Q(a)=a_1^2+a_2^2+a_3^2$.
 The $\text{SL}(2,\C)$-orbit $Q(a)=0$ is parametrized by
$\left(u_{1}^2-u_{2}^2,2u_{1}u_{2},\sqrt{-1}(u_{1}^2+u_{2}^2)\right)^{T}$ where $(u_{1}, u_{2})^{T} \in \C^{2}$.
Let: 
$$M=\begin{pmatrix}m_{11}&m_{12}&m_{13}\\
m_{21}&m_{22}&m_{23}\\
m_{31}&m_{32}&m_{33}\end{pmatrix}\in {\C}^{3\times 3},$$

\noindent and define the $4\times 4$ traceless symmetric matrix $s(M)$ (depending linearly on $M$):
{\footnotesize
\begin{align}\label{eqn:isometry}
s(M):=
\bgroup\frac{1}{2}\begin{pmatrix}
\\[-10pt]
{m}_{11}-{m}_{22}-{m}_{33}&
       {m}_{13}+{m}_{31}&
       {m}_{12}+{m}_{21}&
       {m}_{23}-{m}_{32}\\[2pt]
       {m}_{13}+{m}_{31}&
       -{m}_{11}-{m}_{22}+{m}_{33}&
       {m}_{23}+{m}_{32}&
       {m}_{12}-{m}_{21}\\[2pt]
       {m}_{12}+{m}_{21}&
       {m}_{23}+{m}_{32}&
       -{m}_{11}+{m}_{22}-{m}_{33}&
       -{m}_{13}+{m}_{31}\\[2pt]
       {m}_{23}-{m}_{32}&
       {m}_{12}-{m}_{21}&
       -{m}_{13}+{m}_{31}&
       {m}_{11}+{m}_{22}+{m}_{33}\\
       \\[-10pt]
       \end{pmatrix}.\egroup
\end{align}}

\noindent This construction furnishes a new view on the essential variety $\mathcal{E}$, as described in Proposition \ref{prop:sm}.

\begin{prop}\label{prop:sm}
The linear map $s$ in \textup{($\ref{eqn:isometry}$)} is a real isometry from the space of $3 \times 3$ real matrices
to the the space of traceless symmetric $4 \times 4$ real matrices.  We have that:
$$M \in \mathcal{E} \,\, \Longleftrightarrow \,\, \textup{rk}(s(M)) \le 2.$$
The complexification of $s$, denoted again by $s$, satisfies
for any $M \in \C^{3 \times 3}$:

$$M \in {\textup{Sing}}(\mathcal{E}_{\C}) \,\, \Longleftrightarrow \,\, \textup{rk}(s(M)) \le 1,$$
$$M \in \mathcal{E_{\C}} \,\, \Longleftrightarrow \,\, \textup{rk}(s(M)) \le 2.$$

\end{prop}
\begin{proof} 
We construct the correspondence over $\C$ at the level of $\text{Sing}(\mathcal{E}_{\C})$
and then we extend it by linearity.
Choose coordinates $(u_{1},u_{2})$ in $U$ and coordinates $(v_{1},v_{2})$ in $V$.
Consider the following parametrization of matrices $M \in \text{Sing}(\mathcal{E}_{\C})$:
\begin{equation}
M \, = \, \begin{pmatrix}
\\[-10pt] 
u_{1}^2-u_{2}^2 \\[2pt] 
2u_{1}u_{2} \\[2pt] 
\sqrt{-1}(u_{1}^2+u_{2}^2) \\
\\[-10pt] 
\end{pmatrix}
\cdot\left(v_{1}^2-v_{2}^2, \, 2v_{1}v_{2}, \, \sqrt{-1}(v_{1}^2+v_{2}^2)\right).\label{eq:mrk1}\end{equation}
Consider also the following parametrization of the Euclidean quadric in $U \otimes V$:
$$k=\left(\sqrt{-1}(u_{2}v_{2}-u_{1}v_{1}), \,
     u_{1}v_{1} +u_{2}v_{2}, \, 
     -\sqrt{-1}( u_{1}v_{2} +u_{2}v_{1} ), \,
     -u_{1}v_{2} +u_{2}v_{1}\right).$$ 
The variety of rank 1 traceless $4\times 4$ symmetric matrices
is accordingly parametrized by $k^{T}k$.
Substituting (\ref{eq:mrk1}) into the right-hand side below, a computation verifies that:
$$k^{T}k = s(M).$$

\noindent This proves the second equivalence in the statement above and explains the definition of $s(M)$,
namely that it is the equivariant embedding from Proposition \ref{prop:eddegree} in coordinates. 
The third equivalence follows because $\mathcal{E}_{\C} = \sigma_{2}(\text{Sing}(\mathcal{E}_{\C}))$,
by Proposition \ref{prop:singessential}(ii).
For the first equivalence, we note that $s$ is defined over $\RR$ and now a direct computation verifies that $\textup{tr}\left(s(M)s(M)^{T}\right) = \textup{tr}\left(MM^{T}\right)$
for $M \in \RR^{3 \times 3}$.
\end{proof}

Note that the ideal of $3$-minors of $s(M)$ is indeed generated by the ten cubics in (\ref{eq:ess}).

\begin{remark}
The critical points of the distance function from any data point $M \in \RR^{3 \times 3}$ to ${\mathcal E}$
can be computed by means of the SVD of $s(M)$, as in \textup{\cite[{Example} 2.3]{DHOST}}.
\end{remark}

\section{Ulrich sheaves on the variety of symmetric $4\times 4$ matrices of rank $\le 2$}\label{sec:Ulrich}

 Our goal is to construct the Chow form of the essential variety.
By the theory of Eisenbud and Schreyer \cite{ESW}, this can be done provided 
one has an Ulrich sheaf on this variety. The 
notions of Ulrich sheaf, Chow forms and the construction of 
\cite{ESW} will be explained below. 

As shown in \S \ref{sec:determinantal}, the essential variety $\mathcal{E}_{\C}$ is a linear section of the 
projective variety
$PX^s_{4,2}$ of symmetric $4 \times 4$ matrices of rank $\leq 2$. If we
construct an Ulrich sheaf on $PX^s_{4,2}$, then a quotient of 
this sheaf by a linear form
is an Ulrich sheaf on $\mathcal{E}_{\C}$ provided
that linear form is regular for the Ulrich sheaf on $PX^s_{4,2}$. 
We will achieve this twice, in $\S \ref{subsec:firstUlrich}$ and $\S \ref{subsec:secondUlrich}$.

\subsection{Definition of Ulrich modules and sheaves}

\begin{definition} \label{def:UlrichUlrich}
A graded module $M$ over a polynomial ring $A = \kr[x_0, \ldots, x_n]$
is an  \textup{Ulrich module} provided:
\begin{enumerate}
\item It is generated in degree $0$ and has a linear minimal free resolution:
\begin{equation} \label{eq:UlrichLinres}
 0 \vmto{} M \leftarrow A^{\beta_0} \vmto{} A(-1)^{\beta_1} \vmto{} 
A(-2)^{\beta_2} \lmto{d_2} \cdots \vmto{} A(-c)^{\beta_c} \vmto{} 0.
\end{equation}
\item The length of the resolution $c$ equals the codimension of
the support of the module $M$.
\item[2'] \hspace{-0.3cm}. The Betti numbers are $\beta_i = \binom{c}{i} \beta_0$
for $i = 0, \ldots, c$. 
\end{enumerate}
One can use either (1) and (2), or equivalently, (1) and (2)' as the definition.
\end{definition}

A sheaf $\cF$
on a projective space $\PP^{n}$ with support of dimension $\geq 1$ 
is an {\it Ulrich sheaf}
provided it is the sheafification of an Ulrich module. Equivalently,
the module of twisted global sections
$M = \bigoplus_{d \in \ZZ} H^{0}(\PP^{n}, \cF(d))$ 
is an Ulrich module over the polynomial ring $A$.
\begin{fact}\label{fact:UlrichRank}
If the support of an Ulrich sheaf $\cF$ is a variety $X$ of degree $d$,
then $\beta_0$ is a multiple of $d$, say $rd$. This corresponds
to $\cF$ being a sheaf of rank $r$ on $X$.
\end{fact}
Since there is a one-to-one correspondence between Ulrich modules
over $A$ and Ulrich sheaves
on $\PP^n$, we interchangably speak of both. But in our constructions
we focus on Ulrich modules.  
A prominent conjecture of Eisenbud and Schreyer \cite[p.543]{ESW}
states that on any variety $X$ in a projective space, there is an Ulrich sheaf
whose support is $X$. 

\subsection{The variety of symmetric $4 \times 4$ matrices}\label{subsec:sym44}
We fix notation.  Let $X^s_4$ be the space of symmetric $4 \times 4$ matrices over
the field $\kr$. This identifies as $\kr^{10}$. 
Let $x_{ij} = x_{ji}$ be 
the coordinate functions on $X^s_4$ where $1 \leq i \leq j \leq 4$,
so the coordinate ring of $X^s_4$ is: $$A = \kr[x_{ij}]_{1 \leq i \leq j \leq 4}.$$ 
For $0 \leq r \leq 4$, denote by $X^s_{4,r}$ the affine subvariety of $X^s_4$ consisting of matrices
of rank $\leq r$. The ideal of $X^s_{4,r}$ is generated by the
$(r+1) \times (r+1)$-minors of the generic $4 \times 4$ 
symmetric matrix $(x_{ij})$.  This
is in fact a prime ideal, by \cite[Theorem 6.3.1]{Wey}.
The rank subvarieties
have the following degrees and codimensions:

\medskip

\begin{centering}

\begin{tabular}{|l|l|l|}
\hline variety & degree & codimension \\
\hline $X^s_{4,4}$ & 1 & 0 \\
\hline $X^s_{4,3}$ & 4 & 1 \\
\hline $X^s_{4,2}$ & 10 & 3 \\
\hline $X^s_{4,1}$ & 8 & 6 \\
\hline $X^s_{4,0}$ & 1 & 10 \\
\hline
\end{tabular}

\end{centering}

\medskip
\noindent Since the varieties
$X^s_{4,r}$ are defined by homogeneous ideals, they
give rise to projective varieties $PX^s_{4,r}$ in the projective
space $\PP^9$.  However, in \S \ref{subsec:firstUlrich} and \S \ref{subsec:secondUlrich}
it will be convenient to work with
affine varieties, and general (instead of special)
linear group actions.

The group $\textup{GL}(4, \C)$ acts on $X^s_4$. If $M \in \textup{GL}(4, \C)$ 
and $X \in X^s_4$, the action is as follows:
\[ M_{\cdot} X = M \cdot X \cdot M^{T}. \]
Since any symmetric matrix
can be diagonalized by a unitary coordinate change,
there are five orbits of the action of $\textup{GL}(4,\C)$ on $X^s_4$, one
per rank of the symmetric matrix.  Let:
$$E=\C^4$$ be a four-dimensional complex vector space.  
The coordinate ring
of $X^s_4$ identifies as $A \iso \Sym(S_{2}(E))$. 
The space of symmetric matrices $X^s_4$ may
then be identified with the dual space $S_2(E)^*$,
so again we see that $\textup{GL}(E) = \textup{GL}(4, \C)$ acts on $S_2(E)^*$.

\subsection{Representations and Pieri's rule}
We shall recall some basic representation theory of the general 
linear group $\textup{GL}(W)$, where $W$ is a $n$-dimensional complex vector space.
The irreducible representations of $\textup{GL}(W)$ are given by Schur modules
$S_\lambda(W)$ where $\lambda$ is a partition: a sequence of integers
$\lambda_1 \geq \lambda_2 \geq \cdots \geq \lambda_n$. 
When $\lambda = d,0, \ldots, 0$, then $S_\lambda(W)$ is the 
$d^{\textup{th}}$ symmetric power $S_{d}(W)$. When $\lambda = 1,\ldots, 1,0, \ldots, 0$, with
$d$ $1$'s, then $S_\lambda(W)$ is the exterior wedge $\wedge^d W$. For~all~partitions~$\lambda$ there are isomorphisms of $\textup{GL}(W)$-representations: 
\[ S_{\lambda}(W)^{*}  \iso S_{-\lambda_{n}, \ldots, -\lambda_{1}}(W) 
\text{\quad \quad and \quad \quad} 
S_\lambda(W) \te (\wedge^n W)^{\te r} \iso S_{\lambda + r \cdot \bfone}(W)
\]
where $\bfone = 1,1,\ldots, 1$. Here  $\wedge^n W$ is the one-dimensional
representation $\C$ of $\textup{GL}(W)$ where a linear map 
$\phi$ acts by its determinant.

Denote by $|\lambda| := \lambda_1 + \cdots + \lambda_n$. Assume $\lambda_n, \, \mu_n \geq 0$.
The tensor product of two Schur modules $S_\lambda(W) \te S_\mu(W)$
splits into irreducibles as
a direct sum of Schur modules:
\[ \bigoplus_{\nu} u(\lambda,\mu; \nu) S_\nu (W) \]
where the sum is over partitions with $|\nu| = |\mu| + |\lambda|$. 
The multiplicities $u(\lambda,\mu;\nu) \in \mathbb{Z}_{\geq 0}$ are determined by the
Littlewood-Richardson rule \cite[Appendix A]{FH}. 
In one case, that will be important to
us below, there is a particularly
nice form of this rule. Given two partitions $\lambda^\prime$ and $\lambda$, we
say that $\lambda^\prime / \lambda$ is a {\it horizontal strip} if 
$\lambda^\prime_i \geq \lambda_i \geq \lambda_{i+1}^\prime$.  

\begin{fact}[Pieri's rule] As $\textup{GL}(W)$-representations, we have the rule:
\[ S_\lambda(W) \te S_d(W) \,\,\,\,\,\,\, \iso \bigoplus_{\overset{|\lambda^\prime| \, = \, |\lambda| + d}
{\lambda^\prime / \lambda \textup{ is a horizontal strip}}} 
\!\!\!\!\!\!\!\!\!\!\!\! S_{\lambda^\prime}(W). \]
\end{fact}

\subsection{The first Ulrich sheaf}\label{subsec:firstUlrich}
We are now ready to describe our first Ulrich sheaf on the projective variety 
$PX^2_{4,2}$. We construct it as an Ulrich module supported on
the variety $X^s_{4,2}$.  We use notation from \S \ref{subsec:sym44},
so $E$ is $4$-dimensional. 
Consider $S_3(E) \te S_2(E)$. By Pieri's rule this decomposes as:
\[ S_5(E) \oplus S_{4,1}(E) \oplus S_{3,2}(E). \]

We therefore get a $\textup{GL}(E)$-inclusion 
$ S_{3,2}(E) \rightarrow S_3(E) \te S_2(E)$ 
unique up to nonzero scale.  
Since $A_{1} = S_{2}(E)$ from \S \ref{subsec:sym44}, this extends uniquely to an $A$-module map:
\[ S_3(E) \te A \lmto{\alpha} S_{3,2}(E) \te A(-1). \]
This map can easily be programmed using \texttt{Macaulay2} and the package 
\texttt{PieriMaps} \cite{Sam}:

\begin{verbatim}
R=QQ[a..d]
needsPackage "PieriMaps"
f=pieri({3,2},{2,2},R)
S=QQ[a..d,y_0..y_9]
a2=symmetricPower(2,matrix{{a..d}})
alpha=sum(10,i->contract(a2_(0,i),sub(f,S))*y_i)
\end{verbatim}

\noindent We can then compute the resolution of the cokernel of $\alpha$ in \texttt{Macaulay2}. It has the form:
\[ A^{20} \lmto{\alpha} A(-1)^{60} \vmto{} A(-2)^{60} \leftarrow A(-3)^{20}. \]
Thus the cokernel of $\alpha$ 
is an Ulrich module by (1) and (2)' in Definition \ref{def:UlrichUlrich}.  
An important point is that
the \texttt{res} command in \texttt{Macaulay2} computes differential matrices in
unenlightening bases.  We completely and intrinsically describe the $\textup{GL}(E)$-resolution below:

\begin{prop} \label{pro:UlrichFirst}
The cokernel of $\alpha$ is an Ulrich module $M$ of rank $2$ 
supported on the variety $X^s_{4,2}$. The resolution of $M$
is $\textup{GL}(E)$-equivariant and it is:
\begin{align} \label{eq:UlrichRes}
F_\dt :  S_3(E) \te A \lmto{\alpha} S_{3,2}(E) \te A(-1)
 & \lmto{\phi} S_{3,3,1}(E) \te A(-2)  \\
  & \lmto{\beta} S_{3,3,3}(E) \te A(-3) \notag
\end{align}
with ranks $20,60,60,20$, and where all differential maps
are induced by Pieri's rule.  The dual complex of this resolution
is also a resolution, and these two resolutions are isomorphic up to twist.
As in \textup{\cite{SW}}, we can visualize the resolution by:
\[
  0 \,\, \gets \,\, M \,\, \vmto{} \,\, \tiny \tableau[scY]{ &  & } \,\, \gets \,\, \tiny \tableau[scY]{ &  & \\ & } \,\, \gets \,\,
   \tiny \tableau[scY]{ & & \\ & & \\ \\} \,\, \gets \,\, \tiny \tableau[scY]{ & & \\ & & \\ & &} \,\, \gets \,\, 0.
  \]
\end{prop}

\begin{proof} 
Since $M$ is the cokernel of a $\textup{GL}(E)$-map, 
it is $\textup{GL}(E)$-equivariant. 
So, the support of $M$ is a union of 
orbits. By Definition \ref{def:UlrichUlrich}(2),
$M$ is supported in codimension $3$.
Since the only orbit of codimension $3$ is $X^s_{4,2}\backslash X^s_{4,3}$,
the support of $M$ is the closure of this orbit, which is
$X^s_{4,2}$. It can also easily be checked with \texttt{Macaulay2}, by
restricting $\alpha$ to diagonal matrices of rank $r$ for $r = 0,\ldots, 4$,
that $M$ is supported on the strata $X^s_{4,r}$ where $r \leq 2$. 
Also, the statement that the rank of $M$ equals 2 is now immediate from
Fact \ref{fact:UlrichRank}.

Now we prove that the $\textup{GL}(E)$-equivariant minimal free resolution of $M$ 
is $F_\dt$ as above. 
By Pieri's rule there is a $\textup{GL}(E)$-map unique up to nonzero scalar:
\[ S_{3,2}(E) \te S_2(E) \vmto{} S_{3,3,1}(E) \]
and a $\textup{GL}(E)$-map unique up to nonzero scalar:
\[ S_{3,3,1}(E) \te S_2(E) \vmto{} S_{3,3,3}(E). \]
These are the maps $\phi$ and $\beta$ in $F_\dt$ respectively. The composition
$\alpha \circ \phi$ maps $S_{3,3,1}(E)$ to a submodule of 
$S_3(E) \te S_2(S_2(E))$. By \cite[Proposition 2.3.8]{Wey} 
the latter double symmetric power
equals $S_4(E) \oplus S_{2,2}(E)$, and so this tensor product decomposes as:
\[ S_3(E) \te S_4(E) \, \bigoplus \, S_3(E) \te S_{2,2}(E). \]
By Pieri's rule, none of these summands contains
$S_{3,3,1}(E)$. Hence $\alpha \circ \phi$ is zero by Schur's lemma.
The same type of argument shows that $\phi \circ \beta$ is zero. Thus
$F_\dt$ is a complex. 

By our \texttt{Macaulay2} computation of Betti numbers before the Proposition, 
$\ker(\alpha)$ is generated in degree $2$ by $60$ minimal generators. In $F_\dt$
these must be the image of $S_{3,3,1}(E)$, since that is $60$-dimensional by the hook 
content formula and it maps injectively to $F_1$.  So $F_{\dt}$ 
is exact at $F_{1}$.
Now again by the \texttt{Macaulay2} computation, it follows that
$\ker \phi$ is generated in degree $3$ by $20$ generators.  These
must be the image of $S_{3,3,3}(E)$ since that is $20$-dimensional
and maps injectively to $F_2$.  So $F_{\dt}$ is
exact at $F_{2}$.  Finally, the computation implies that
$\beta$ is injective, and $F_\dt$ is the $\textup{GL}(E)$-equivariant minimal free resolution of
$M$.

For the statement about the dual, recall that since $F_{\dt}$ is a resolution of a Cohen-Macaulay module, 
the dual complex, obtained by applying $\Hom_{A}(-, \omega_A)$ with $\omega_A = A(-10)$, 
is also a resolution. 
If we twist this dual resolution with $(\wedge^4 E)^{\te 3} \te A(7)$, the 
terms will be as in the original resolution. Since the nonzero $\textup{GL}(E)$-map 
$\alpha$ is uniquely determined up to scale, it follows that $F_{\dt}$
and its dual are isomorphic up to twist.
\end{proof}

\begin{remark}
The $\textup{GL}(E)$-representations in this resolution could also 
have been computed using the \texttt{Macaulay2} package \texttt{HighestWeights} \cite{Ga}.
\end{remark}

\begin{remark} \label{rem:anton}
The dual of this resolution is:
\begin{equation} \label{eq:DualUlrich}
S_{3,3,3}(E^*) \te A \leftarrow S_{3,3,1} \te A(-1) \leftarrow
S_{3,2}(E^*) \te A(-2) \leftarrow S_{3}(E^*) \te A(-3).
\end{equation}
A symmetric form $q$ in $S_2(E^*)$ corresponds to a point in $\textup{Spec}(A)$
and a homomorphism $A \rightarrow {\mathbb C}$. The fiber of this complex
over the point $q$ is then an $\textup{SO}(E^*, q)$-complex:
\begin{equation} \label{eq:LW} S_{3,3,3}(E^*) \leftarrow S_{3,3,1} \leftarrow
S_{3,2}(E^*) \leftarrow S_{3}(E^*).
\end{equation}
When $q$ is a nondegenerate form, this is the {\it Littlewood complex}
$L^{3,3,3}_\bullet$ as defined in \cite[\S 4.2]{SSW}. (The terms
of $L^{3,3,3}$ can be computed using the plethysm in \S 4.6 of loc.cit.)
This partition $\lambda = (3,3,3)$ is not admissible since $3+3 > 4$,
see Sec.4.1 loc.cit. The cohomology of \eqref{eq:LW} is then
given by Theorem 4.4 in loc.cit. and it vanishes (since here $i_4(\lambda) = 
\infty$), as it should in agreement with Proposition \ref{pro:UlrichFirst}.
The dual resolution \eqref{eq:DualUlrich} of the Ulrich sheaf can then 
be thought of as a ``universal" Littlewood complex for the parition $\lambda
= (3,3,3)$.  In other cases when Littlewood complexes
are exact, it would be an interesting future research topic to investigate the sheaf that
is resolved by the ``universal Littlewood complex".
\end{remark}

To obtain nicer
formulas for the Chow form of the essential variety $\mathcal{E}_{\C}$ 
in \S \ref{sec:Chow}, 
we now prove that the middle map $\phi$ in the resolution \textup{(\ref{eq:UlrichRes})} is symmetric,
in the following appropriate sense. 
In general, suppose that we are given a linear map $W^* \mto{\mu} W \te L^*$
where $L$ is a finite dimensional vector space. Dualizing, we get a map
$W \vmto{\mu^{T}} W^* \te L$ which in turn gives a map
$W \te L^* \vmto{\nu} W^*$. By definition, the map $\mu$ is {\it symmetric} if 
$\mu = \nu$ and {\it skew-symmetric} if $\mu = - \nu$.  If $\mu$ is symmetric and $\mu$ is represented as a matrix 
with entries in $L^*$ with respect to dual bases of $W$ and $W^*$, then that matrix is symmetric, and analogously 
when $\mu$ is skew-symmetric.  Note that the map $\mu$ also induces a map $L  \mto{\eta} W \te W$. 

\begin{fact}\label{fact:sym} The map $\mu$
is symmetric if the image of $\eta$ is in the subspace 
$S_{2}(W) \subseteq W \te W$ and it is skew-symmetric if the image is in the subspace
$\wedge^2 W \subseteq W \te W$. 
\end{fact}

\begin{prop}\label{symM}
The middle map $\phi$ in the resolution \textup{(\ref{eq:UlrichRes})} is symmetric.
\end{prop}

\begin{proof}
Consider the map $\phi$ in degree $3$. It is:
\[ S_{3,2}(E) \te S_2(E) \vmto{} S_{3,3,1}(E) \iso S_{3,2}(E)^* 
\te (\wedge^4 E)^{\te 3} \]
and it induces the map:
\[ S_{3,2}(E) \te S_{3,2}(E) \vmto{} S_2(E)^* \te (\wedge^4 E)^{\te 3} 
\iso S_{3,3,3,1}(E). \]
By the Littlewood-Richardson rule, the right representation
above occurs with multiplicity $1$ in the left side.
Now one can check that $S_{3,3,3,1}(E)$ occurs in $S_2(S_{3,2}(E))$.
This follows by Corollary 5.2 in \cite{BeLec} or one can use 
the package \texttt{SchurRings} \cite{MaSchur} in \texttt{Macaulay2}: 

\begin{verbatim}
needsPackage "SchurRings"
S = schurRing(s,4,GroupActing=>"GL")
plethysm(s_2,s_{3,2}) 
\end{verbatim}

\noindent Due to Fact \ref{fact:sym}, we can conclude that the map $\phi$ is symmetric.
\end{proof}

\subsection{The second Ulrich sheaf}\label{subsec:secondUlrich}
We construct another Ulrich sheaf on $PX^s_{4,2}$ and analyze it similarly to as above.
This will lead to a second formula for $\text{Ch}(\mathcal{E}_{\C})$ in 
\S \ref{sec:Chow}.
Consider $S_{2,2,1}(E) \te S_2(E)$. By Pieri's rule:
\[ S_{2,2,1}(E) \te S_2(E) \, \iso \, S_{4,2,1}(E) \oplus S_{3,2,2}(E) \oplus S_{3,2,1,1}(E) 
\oplus S_{2,2,2,1}(E). \]
Thus there is a $\textup{GL}(E)$-map, with nonzero degree 1 components unique up to scale: 
\[ S_{2,2,1}(E) \te A \lmto{\alpha} (S_{3,2,2}(E) \oplus S_{3,2,1,1}(E)
\oplus S_{2,2,2,1}(E)) \te A(-1). \]

\noindent This map can be programmed in \texttt{Macaulay2} using \texttt{PieriMaps} as follows:

\begin{verbatim}
R=QQ[a..d]
needsPackage "PieriMaps"
f1= transpose pieri({3,2,2,0},{1,3},R)
f2=transpose pieri({3,2,1,1},{1,4},R)   
f3=transpose pieri({2,2,2,1},{3,4},R) 
f = transpose (f1||f2||f3)
S=QQ[a..d,y_0..y_9]
a2=symmetricPower(2,matrix{{a..d}})
alpha=sum(10,i->contract(a2_(0,i),sub(f,S))*y_i)
\end{verbatim}

\noindent We can then compute the resolution of $\textup{coker}(\alpha)$ in \texttt{Macaulay2}. It has the form:
\[ A^{20} \lmto{\alpha} A(-1)^{60} \vmto{} A(-2)^{60} \vmto{} A(-3)^{20}. \]
Thus the cokernel of $\alpha$ is an Ulrich module, and moreover we have: 

\begin{prop} \label{pro:UlrichSecond}
The cokernel of $\alpha$ is an Ulrich module $M$ of rank $2$ 
supported on the variety $X^s_{4,2}$. The resolution of $M$
is $\textup{GL}(E)$-equivariant and it is:
\begin{align} 
& F_\dt :   S_{2,2,1}(E) \te A & & 
\lmto{\alpha} (S_{3,2,2}(E) \oplus S_{3,2,1,1}(E) \oplus S_{2,2,2,1}(E)) 
\te A(-1) \notag\\
& & & \lmto{\phi} (S_{4,2,2,1}(E) \oplus S_{3,3,2,1}(E) \oplus S_{3,2,2,2}(E))
\te A(-2)  \label{eq:UlrichRes2}\\
& & &  \lmto{\beta} S_{4,3,2,2}(E) \te A(-3) \notag
\end{align}
with ranks $20,60,60,20$.
The dual complex of this resolution is also a resolution and
these two resolutions are isomorphic up to twist.  We can visualize the 
resolution by:
\[
0 \,\, \gets \,\, M \,\, \vmto{} \,\, \tiny \tableau[scY]{ & \\ & \\ \\} 
\,\,  \gets \,\, \tiny \tableau[scY]{ & & \\ & \\ &} \, \oplus \, \tiny \tableau[scY]{ & & \\ & \\ \\ \\} 
\, \oplus \, \tiny \tableau[scY]{ & \\ & \\ & \\ \\}
\,\, \gets \,\, \tiny \tableau[scY]{ & & & \\ & \\ & \\ \\} \, \oplus \, \tiny \tableau[scY]{ & & \\ & & \\ & \\ \\} 
\, \oplus \, \tiny \tableau[scY]{ & & \\ & \\ & \\ &}
\,\, \gets \,\, \tiny \tableau[scY]{ & & & \\ & & \\ & \\ &}
\,\, \gets \,\, 0.
\]
\end{prop}

\begin{proof} 
The argument concerning the support of $M$ is exactly as
in Proposition \ref{pro:UlrichFirst}.

Now we prove that the minimal free resolution of $M$ is of the form above, differently than 
in \ref{pro:UlrichFirst}. 
To start, note that the module $S_{4,2,2,1}(E)$ occurs by Pieri once in
each of:
\[ S_{3,2,2}(E) \te S_2(E), \quad S_{3,2,1,1}(E) \te S_2(E), \quad
S_{2,2,2,1}(E) \te S_2(E). \]
On the other hand, it occurs in: 
\[ S_{2,2,1}(E) \te S_2(S_2(E)) \iso S_{2,2,1}(E) \te S_4(E) \oplus
S_{2,2,1}(E) \te S_{2,2}(E) \]
only twice, as seen using Pieri's rule and the Littlewood-Richardson rule.
Thus $S_{4,2,2,1}(E)$ occurs at least once in the degree $2$ part of 
$\ker(\alpha)$. Similarly we see that each of $S_{3,3,2,1}(E)$ and
$S_{3,2,2,2}(E)$ occurs at least
once in $\ker(\alpha)$ in degree $2$. 
But by the \texttt{Macaulay2} computation before this Proposition, 
we know that 
$\ker(\alpha)$ is a module with $60$ generators in 
degree $2$. And the sum of the dimensions of these three representations 
is $60$. Hence each of them occurs exactly once in $\ker(\alpha)$ in
degree $2$, and they generate $\ker(\alpha)$. 

Now let $C$ be the $20$-dimensional vector space generating 
$\ker(\phi)$. Since the resolution of $M$ has length
equal to $\textup{codim}(M)$, the module $M$ is Cohen-Macaulay and
the dual of its resolution, obtained by applying $\Hom_A(-, \omega_A)$
where $\omega_A \iso A(-4)$,  is again a resolution of 
$\Ext_{A}^3(M, \omega_A)$.  Thus the map from $C \te A(-3)$ to
each of: 
\[ S_{4,2,2,1}(E) \te A(-2), \quad S_{3,3,2,1}(E) \te A(-2), \quad
S_{3,2,2,2}(E) \te A(-2) \]
is nonzero.  In particular $C$ maps nontrivially to: 
\[ S_{3,2,2,2}(E) \te S_2(E) \iso S_{5,2,2,2}(E) \oplus S_{4,3,2,2}(E). \]
Each of the right-hand side representations have dimension $20$, so one of them 
equals $C$. However only the last one occurs in $S_{3,3,2,1}(E) \te S_2(E)$,
and so $C \iso S_{4,3,2,2}(E)$. 
We have proven that the $\textup{GL}(E)$-equivariant minimal 
free resolution of $M$ indeed has the form $F_{\dt}$.

For the statement about the dual, recall that each of the three components
of $\alpha$ in degree 1 are nonzero.
Also, as the dual complex is a resolution,
here obtained by applying $\Hom_A(-, \omega_A)$ with $\omega_A = A(-10)$,
all three degree 1 components of $\beta$ are nonzero.
If we twist this dual resolution with $(\wedge^4 E)^{\te 4} \te A(7)$, the 
terms will be as in the original resolution. Because each of the three
nonzero components of the map 
$\alpha$ are uniquely determined up to scale, 
the resolution $F_{\dt}$
and its dual are isomorphic up to twist.
\end{proof}

\begin{remark}
Again the $\textup{GL}(E)$-representations in this resolution could 
have been computed using the \texttt{Macaulay2} package \texttt{HighestWeights}.
\end{remark}

\begin{prop}\label{symN}
The middle map $\phi$ in the resolution \textup{(\ref{eq:UlrichRes2})} is symmetric.
\end{prop}

\begin{proof}
We first show that the three `diagonal' components of $\phi$ in \textup{(\ref{eq:UlrichRes2})} are symmetric:
\begin{align*}
S_{3,2,2}(E) \te S_2(E) & \vmto{\phi_1}  S_{4,2,2,1}(E) & \\
S_{3,2,1,1}(E) \te S_2(E) & \vmto{\phi_2}  S_{3,3,2,1}(E) & \\
S_{2,2,2,1}(E) \te S_2(E) & \vmto{\phi_3}  S_{3,2,2,2}(E). & 
\end{align*}

\noindent Twisting the third component $\phi_3$ with $(\wedge^{4} E^{*})^{\te 2}$, it identifies as: 
\[ E^* \te S_2(E) \vmto{} E \]
and so $\phi_{3}$ is obviously symmetric.
Twisting the second map $\phi_2$ with $\wedge^4 E^{*}$ it identifies as:
\[ S_{2,1}(E) \te S_2(E) \vmto{} S_{2,2,1}(E) = (S_{2,1}(E)^*) \te 
(\wedge^4 E)^{\te 2}, \]
which induces the map:
\[ S_{2,1}(E) \te S_{2,1}(E) \vmto{} S_2(E)^* \te (\wedge^4 E)^{\te 2} 
= S_{2,2,2}(E). \]
By the Littlewood-Richardson rule, the left tensor product
contains $S_{2,2,2}(E)$ with multiplicity $1$. By Corollary 5.2 in
\cite{BeLec} or \texttt{SchurRings} in \texttt{Macaulay2},
this is in $S_2(S_{2,1}(E))$:

\begin{verbatim}
needsPackage "SchurRings"
S = schurRing(s,4,GroupActing=>"GL")
plethysm(s_2,s_{2,1})    
\end{verbatim}

\noindent So by Fact \ref{fact:sym}, the component $\phi_{2}$ is symmetric.  The first map $\phi_1$ may be identified as: 
\[ S_{3,2,2}(E) \te S_2(E) \vmto{} (S_{3,2,2}(E))^* \te (\wedge^4 E)^{\te 4}, \]
which induces the map:
\[ S_{3,2,2}(E) \te S_{3,2,2}(E) \vmto{} S_2(E)^* \te (\wedge^4 E)^{\te 4} =
S_{4,4,4,2}(E). \]
Again by Littlewood-Richardson, $S_{4,4,4,2}(E)$ is contained with 
multiplicity $1$ in the left side.
By Corollary 5.2 in
\cite{BeLec} or the package \texttt{SchurRings} in \texttt{Macaulay2},
this is in $S_2(S_{3,2,2}(E))$:

\begin{verbatim}
needsPackage "SchurRings"
S = schurRing(s,4,GroupActing=>"GL")
plethysm(s_2,s_{3,2,2})    
\end{verbatim}

It is now convenient to tensor the resolution (\ref{eq:UlrichRes2}) by $(\wedge^4 E^*)^{\te 2}$, and to let:
\[ T_1 = S_{1,0,0,-2}(E), \quad T_2 = S_{1,0,-1,-1}(E), \quad T_3 = 
S_{0,0,0,-1}(E). \]
We can then write the middle map as:
\begin{equation} \label{eq:UlrichRes2Mid}
T_1 \te A(1) \oplus T_2 \te A(1) \oplus T_3 \te A(1) 
\vmto{\footnotesize{\left ( \begin{matrix} \phi_1 & \mu_2 & \nu_2 \\
                     \mu_1 & \phi_2 & 0 \\
                     \nu_1 & 0 & \phi_3
      \end{matrix} \right) }}
T_1^* \te A(-1) \oplus T_2^* \te A(-1) \oplus T_3^* \te A(-1) 
\end{equation} 
Note indeed that the component: 
\[ S_{1,0,-1,-1}(E) \te S_2(E) = T_2 \te S_2(E) \vmto{}
T_3^* \iso S_{1}(E)
\] must be zero, since the left tensor product does not contain
$S_1(E)$ by Pieri's rule. Similarly the map
$T_3 \te S_2(E) \vmto{} T_2^*$ is zero.

We know the maps $\phi_1, \phi_2$ and $\phi_3$ are symmetric.
Consider:
\[ T_2 \te A(1) \vmto{\mu_1} T_1^* \te A(-1), \quad
T_1 \te A(1) \vmto{\mu_2} T_2^* \te A(-1). \]
Since the resolution (\ref{eq:UlrichRes2}) is
isomorphic to its dual, either both $\mu_1$ and $\mu_2$ are nonzero,
or they are both zero.
Suppose both are nonzero. The dual of $\mu_2$ is 
$T_2 \te A(1) \vmto{\mu_2^T} T_1^* \te A(-1)$. 
But such a $\textup{GL}(E)$-map is unique up to scalar, as is
easily seen by Pieri's rule.
Thus whatever the case we can say that $\mu_1 = c_\mu \mu_2^{T}$ 
for some nonzero scalar $c_\mu$. 
Similarly we get $\nu_1 = c_\nu \nu_{2}^{T}$. Composing the map
(\ref{eq:UlrichRes2Mid}) with the automorphism on its
right given by the block matrix:
\[ \left ( \begin{matrix} 1 & 0 & 0 \\
                0 & c_\mu & 0 \\
               0 & 0 & c_\nu 
   \end{matrix} \right ),
\] we get a middle map:
\begin{equation*} 
T_1 \te A(1) \oplus T_2 \te A(1) \oplus T_3 \te A(1) 
\vmto{\footnotesize{\left (\begin{matrix} \phi_1 & \mu_2^\prime & \nu_2^\prime \\
                     \mu_1 & \phi_2^\prime & 0 \\
                     \nu_1 & 0 & \phi_3^\prime
      \end{matrix} \right )}}
T_1^* \te A(-1) \oplus T_2^* \te A(-1) \oplus T_3^* \te A(-1) 
\end{equation*} 
where the diagonal maps are still symmetric, and 
$\mu_1 = (\mu_2^\prime)^T$ and $\nu_1 = (\nu_2^\prime)^T$. 
So we get a symmetric map, and the result about $\phi$ follows.
\end{proof}

This second Ulrich module constructed above in Proposition \ref{pro:UlrichSecond}
is a particular instance of a general 
construction of Ulrich modules on the variety of symmetric $n \times n$
matrices of rank $\leq r$; see \cite{Wey}, \S 6.3 and Exercise 34 in \S 6. 
We briefly
recall the general construction.
Let $W = \C^{n}$ and $G$ be
the Grassmannian $\textup{Gr}(n-r,W)$ of $(n-r)$-dimensional subspaces of $W$.
There is a tautological exact sequence of algebraic vector bundles on $G$:
\[ 0 \rightarrow \cK \rightarrow W \te \cO_G \rightarrow \cQ \rightarrow 0, \]
where $r$ is the rank of $\cQ$.
Let $X = X^s_n$ be the affine space of symmetric $n \times n$ matrices, and 
define $Z$ to be the incidence subvariety of $X \times G$ given by:
\[ Z = \{ ((W \mto{\phi} W),(\C^{n-r} \overset{i}{\hookrightarrow} W)) 
\in X \times G \, | \, \phi \circ i = 0 \}. \]
The  variety $Z$ is the affine geometric bundle $\VV_G(S_2(\cQ))$ 
of the locally free sheaf 
$S_2 (\cQ)$ on the Grassmannian $G$. There is a commutative
diagram: 
\[ \begin{CD}
Z @>>> X \times G \\
@VVV @VVV \\
X^s_{n,r} @>>> X
\end{CD} \]
in which $Z$ is a desingularization of $X^s_{n,r}$.
For any locally free sheaf $\cE$, the Schur functor $S_\lambda$
applies to give a new locally free sheaf $S_\lambda(\cE)$. 
Consider then the locally free sheaf:
\[ \cE(n,r) = S_{(n-r)^r}(\cQ) \te S_{n-r-1,n-r-2, \cdots, 1,0}(\cK) \]
on the Grassmannian $\textup{Gr}(n-r,W)$. Note that $S_{(n-r)^r}(\cQ)=\left(\det(\cQ)\right)^{n-r}$ is a line bundle and $\cE(n,r)$ is a locally free sheaf of 
rank $2^{\binom{n-r}{2}}$.
Let $Z \mto{p} G$ be the projection map. By pullback we get the locally 
free sheaf 
$p^* (\cE(n,r))$ on $Z$. The pushforward of this locally free sheaf down
to $X^s_{n,r}$ is an Ulrich sheaf on this variety. Since $X^s_{n,r}$ is
affine this corresponds to the module of global sections $H^{0}(Z,p^* \cE)$. The Ulrich
module in Proposition \ref{pro:UlrichSecond} is that module when $n = 4$
and $r = 2$.  For our computational purposes realized in \S \ref{sec:Chow}, we worked 
out the equivariant minimal free resolution as above.  
Interestingly, we do not know yet whether
the `simpler' Ulrich sheaf presented in \S \ref{subsec:firstUlrich}, which is new to our knowledge,
generalizes to a construction for other varieties.

\section{The Chow form of the essential variety}
\label{sec:Chow}

\subsection{Grassmannians and Chow divisors}
The Grassmannian variety $\textup{Gr}(c,n+1) = \textup{Gr}(\PP^{c-1}, \PP^{n})$ parametrizes the linear subspaces
of dimension $c-1$ in $\PP^n$, i.e the $\PP^{c-1}$'s in $\PP^n$. 
Such a linear subspace may be
given as the rowspace of a $c \times (n+1)$ matrix.
The tuple of maximal minors of this matrix is uniquely determined
by the linear subspace up to scale.  The number of such minors
is $\binom{n+1}{c}$. Hence we get a well-defined point in the projective space
$\PP^{\binom{n+1}{c} - 1}$. This defines an embedding of the
Grassmannian $\textup{Gr}(c,n+1)$ into that projective space, called the
Pl\"ucker embedding. Somewhat more algebraically, let $W$ be a vector space
of dimension $n+1$ and let $\PP(W)$ be
the space of lines in $W$ through the origin.  
Then a linear subspace $V$ of dimension $c$ in $W$ 
defines a line $\wedge^{c} V$ in $\wedge^{c} W$, and 
so it defines a point in $\PP(\wedge^{c} W) = \PP^{\binom{n+1}{c} - 1}$.
Thus the Grassmannian $\textup{Gr}(c,W)$ embeds into $\PP(\wedge^{c}W)$.
 
If $X$ is a variety of codimension $c$ in a projective space $\PP^n$, then
a linear subspace of dimension $c-1$ will typically not intersect $X$.
The set of points in the Grassmannian $\textup{Gr}(c,n+1)$ that do have
nonempty intersection with $X$ forms a divisor in $\textup{Gr}(c,n+1)$,
called the {\it Chow divisor}. The divisor
class group of $\textup{Gr}(c,n+1)$ is isomorphic to $\ZZ$. Considering the Pl\"ucker
embedding  $\textup{Gr}(c,n+1) \subseteq \PP^{\binom{n+1}{c} - 1}$, any hyperplane
in the latter projective space intersects the Grassmannian in a
divisor which generates the divisor class group of $\textup{Gr}(c,n+1)$. 
The homogeneous coordinate ring of this projective space 
$ \PP^{\binom{n+1}{c} - 1} = \PP(\wedge^{c} W)$
is $\Sym(\wedge^{c} W^*)$.  Note that here $\wedge^c W^*$ are the 
linear forms, i.e. the elements of degree $1$.
If $X$ has degree $d$, then its Chow divisor is cut out by a single form 
$\text{Ch}(X)$ of degree $d$ 
unique up to nonzero scale,
called the {\it Chow form,} in the
coordinate ring of the Grassmannian $\Sym(\wedge^{c} W^*)/ I_{\textup{Gr}(c,n+1)}$.
As the parameters $n, c, d$ increase, Chow forms become unwieldy
to even store on a computer file.  Arguably, the most efficient (and useful) representations 
of Chow forms are as determinants or Pfaffians of a matrix with entries in $\wedge^{c} W^*$.  
As we explain next, Ulrich sheaves can give such formulas.

\subsection{Construction of Chow forms}
We now explain how to obtain the Chow form $\textup{Ch}(X)$ of a variety $X$
from an Ulrich sheaf $\cF$ whose support is $X$. The reference
for this is \cite[p.552-553]{ESW}. Let
$M = \oplus_{d \in  \ZZ} H^{0}(\PP^n, \cF(d))$ be the graded module
of twisted global sections over the polynomial ring $A = \kr[x_0, \ldots, x_n]$. 
We write
$W^*$ for the vector space generated by the variables $x_0, \ldots, x_n$.
Consider the minimal free resolution (\ref{eq:UlrichLinres}) of $M$. The
map $d_i$ may be represented by a matrix $D_i$ of size $\beta_i\times
\beta_{i+1}$, with entries in the linear space $W^*$. Since
(\ref{eq:UlrichLinres}) is a complex the product of two successive matrices
$D_{i-1} D_i$ is the zero matrix. Note that when we multiply the entries of these
matrices, we are multiplying elements in the ring $A = \Sym(W^*) = 
\kr[x_0, \ldots, x_n]$.  

Now comes the shift of view: Let $B = \oplus_{i = 0}^n
\wedge^i W^*$ be the exterior algebra on the vector space $W^*$. 
We now consider the entries in the $D_i$ (which are all degree 
one forms in $A_1 = W^* = B_1$) 
to be in the ring $B$ instead.
We then multiply together all the matrices $D_i$ corresponding to 
the maps $d_i$. The multiplications of the entries are performed in the skew-commutative ring
$B$. We then get a product:
\[ D = D_0\cdot D_1 \cdots D_{c-1}, \]
where $c$ is the codimension of the variety $X$ which supports $\cF$. 
If $\cF$ has rank $r$ and the degree of $X$ is $d$, the matrix
$D$ is a nonzero $rd \times rd$ matrix.
The entries in the product $D$ now lie in $\wedge^c W^*$. 
Now comes the second shift of
view: We consider the entries of $D$ to be linear forms in the polynomial ring
$\Sym(\wedge^c W^*)$. Then we take the determinant
of $D$, computed in this polynomial ring, and get a form of degree $rd$ in $\Sym(\wedge^c W^*)$. 
When considered in the coordinate ring
of the Grassmannian $\Sym(\wedge^c W^*)/I_G$, then $\det(D)$ equals the $r^{\text{th}}$ power of the
Chow form of $X$.  
For more information on the fascinating links between the 
symmetric and exterior algebras, the reader can start with
the Bernstein-Gel'fand-Gel'fand correspondence as treated in
\cite{EFS}.

\subsection{Skew-symmetry of the matrices computing the 
Chow form of $PX^s_{4,2}$}
In \S \ref{sec:Ulrich} we constructed two different Ulrich modules of rank 
$2$ on
the variety $PX^s_{4,2}$ of symmetric $4\times 4$ matrices of rank $\leq 2$. 
That variety has degree $10$. The matrix $D$ thus in both cases
is $20 \times 20$, and its determinant
is a square in $\Sym(\wedge^c W^*)$. 
In fact, and here our analysis of the equivariant resolutions pays off,
the matrix $D$ in both cases is skew-symmetric when we use 
the bases distinguished by representation theory for the differential matrices:

\begin{lemma} Let $A,B,C$ be matrices of linear forms in the exterior
algebra. Their products behave as follows under transposition:
\begin{enumerate}
\item $(A \cdot B)^T = -B^T \cdot A^T$
\item $(A \cdot B \cdot C)^T = -C^T \cdot B^T \cdot A^T$.
\end{enumerate}
\end{lemma}

\begin{proof} Part (1) is because $uv = -vu$ when 
$u$ and $v$ are linear forms in the exterior algebra.
Part (2) is because $uvw = -wvu$ for linear forms in the exterior algebra.
\end{proof}

The resolutions (\ref{eq:UlrichRes}) and (\ref{eq:UlrichRes2}) of our
two Ulrich sheaves, have the form:
\begin{equation} \label{eq:FGGF} 
F \vmto{\alpha} G \vmto{\phi} G^* \vmto{\beta} F^*.
\end{equation}
Dualizing and twisting we get the resolution:
\[ F \vmto{\ \beta^T} G \vmto{\ \phi^T} G^* \vmto{\ \alpha^T} F^*. \]
Since $\phi = \phi^T$, both $\beta$ and $\alpha^T$ map isomorphically onto the same image.
We can therefore replace the map $\beta$ in (\ref{eq:FGGF})  with $\alpha^T$, and get the $\textup{GL}(E)$-equivariant resolution:
\[ F \vmto{\alpha} G \vmto{\phi} G^* \vmto{\alpha^T} F^*. \]
Let $\oval, \ovphi$ and $\ovalt$ be the maps in the resolution
above, but now considered to live over the exterior algebra. 
The Chow form associated to the two Ulrich sheaves is then the 
Pfaffian of the matrix:
\[ \oval \, \ovphi \, \ovalt.\]

\begin{prop}\label{SkewSym} The Chow form $\textup{Ch}(PX^s_{4,2})$ constructed from the Ulrich sheaf is, 
in each case, the Pfaffian of a $20 \times 20$ skew-symmetric
matrix.
\end{prop}

\begin{proof} The Chow form squared is the determinant of $\oval \, \ovphi \, \ovalt$
and we have: 

\vspace{0.2cm}

\hspace{3.5cm}
$ \big{(}\oval \, \ovphi \, \ovalt\big{)}^T = - \, (\ovalt)^T \, \ovphi^T \, \ovalt
= - \, \oval \, \ovphi \, \ovalt.$
\end{proof}

\subsection{Explicit matrices computing the Chow form of $PX^s_{4,2}$}

Even though our primary
aim is to compute the Chow form of the essential variety,
we get explicit matrix formulas for the Chow form of $PX^s_{4,2}$
as a by-product of our method.
We carried out the computation in Proposition \ref{SkewSym} in \texttt{Macaulay2} for both Ulrich
modules on $PX^s_{4,2}$.  We used the package \texttt{PieriMaps} to make matrices $D_1$ and $D_2$ 
representing $\alpha$ and $\phi$ with respect to the built-in choice of bases
parametrized by semistandard tableaux.  We had to multiply $D_2$ on the right by a
change of basis matrix to get a matrix representative with respect to dual bases, 
i.e. symmetric.  For example in the case of the first Ulrich module (\ref{eq:UlrichRes})
this change of basis matrix
computes the perfect pairing $S_{3,2}(E) \te S_{3,3,1}(E) \to (\wedge^4 E)^{\te 3}$. 
Let us describe the transposed inverse matrix that represents the dual 
pairing.  Columns are labeled by the
semistandard Young tableaux $S$ of shape $(3,2)$, and rows are labeled
by the semistandard Young tableaux $T$ of shape $(3,3,1)$.
The $(S,T)$-entry in the matrix is obtained by fitting together the tableau $S$ and the tableau $T$ rotated by $180^{\circ}$
into a tableau of shape $(3,3,3,3)$, straightening, and then taking the coefficient of
{\tiny \tableau[scY]{ 0 & 0 & 0 \\ 1 & 1 & 1 \\ 2 & 2 & 2 \\ 3 & 3 & 3}}.
To finish for each Ulrich module, we took the product $D_{1}D_{2}D_{1}^{T}$ over the exterior algebra. 

The two resulting explicit $20 \times 20$ skew-symmetric matrices are available as \texttt{arXiv}
ancillary files or at this paper's
webpage\footnote{\url{http://math.berkeley.edu/~jkileel/ChowFormulas.html}}.
Their Pfaffians equal the Chow form of $PX^s_{4,2}$, which is an element in the
homogeneous coordinate of the $\textup{Gr}(3,10) = \textup{Gr}(\PP^2, \PP^9)$.  To get a feel for the 
`size' of this Chow form, note that this ring is a quotient of the
polynomial ring $\Sym(\wedge^{3}\Sym_{2}(E))$ in 120 Pl\"ucker variables, denoted 
$\Q[p_{\{11,12,13\}}, \ldots, p_{\{33,34,44\}}]$ on our website, by the ideal minimally generated
by 2310 Pl\"ucker quadrics.  We can compute that the degree 10 piece where $\text{Ch}(PX^s_{4,2})$ lives is a 
108,284,013,552-dimensional vector space.  

Both $20 \times 20$ matrices afford extremely compact formulas for this special element.  
Their entries are
linear forms in $p_{\{11,12,13\}}, \ldots, p_{\{33,34,44\}}$ with one- and two-digit relatively prime
integer coefficients.  No more than $5$ of the $p$-variables appear in any entry.  In the first matrix, $96$
off-diagonal entries equal 0.  The matrices give new expressions for one of the two irreducible factors of a 
discriminant studied since 1879 by George Salmon (\cite{Sal}) and as recently as 2011 (\cite{PSV}), as we see next in Remark \ref{rem:salmon}.

\begin{remark}\label{rem:salmon}
From the subject of plane curves, it is classical that every ternary quartic form $f \in \C[x,y,z]_{4}$ 
can be written as $f = \textup{det}(xA + yB + zC)$ for some $4 \times 4$ symmetric matrices 
$A, B, C$.  Geometrically, this expresses $\textup{V}(f)$ inside the 
net of plane quadrics $\langle A, B, C \rangle$ as the locus of singular quadrics.  
By Theorem \textup{7.5} of \textup{\cite{PSV}}, that plane quartic curve $\textup{V}(f)$ is singular if and only if the 
\textup{Vinnikov discriminant}: 
\[ \Delta(A, B, C) = \textup{\textbf{M}}(A,B,C)\textup{\textbf{P}}(A,B,C)^{2} \]
evaluates to 0.  Here $\textup{\textbf{M}}$ is a degree $(16,16,16)$ polynomial 
known as the \textup{tact invariant}
and $\textup{\textbf{P}}$ is a degree $(10,10,10)$ polynomial.
The factor $\textup{\textbf{P}}$ equals the Chow form $\textup{Ch}(PX^s_{4,2})$ 
after substituting Pl\"ucker coordinates for Stiefel coordinates:
\[ p_{\{i_{1}j_{1},i_{2}j_{2},i_{3}j_{3}\}} = 
\textup{det} \begin{pmatrix}
a_{i_{1}j_{1}} & a_{i_{2}j_{2}} & a_{i_{3}j_{3}} \\
b_{i_{1}j_{1}} & b_{i_{2}j_{2}} & b_{i_{3}j_{3}} \\
c_{i_{1}j_{1}} & c_{i_{2}j_{2}} & c_{i_{3}j_{3}} 
\end{pmatrix}.\]
\end{remark}

\subsection{Explicit matrices computing the Chow form of $\mathcal{E_{\C}}$}

We now can put everything together and solve the problem raised 
by Agarwal, Lee, Sturmfels and Thomas in \cite{ALST} of computing the Chow form of the essential variety.
In Proposition \ref{prop:sm}, we constructed a linear embedding $s \colon \PP^{8} \hookrightarrow \PP^{9}$ that restricts
to an embedding $\mathcal{E}_{\C} \hookrightarrow PX^s_{4,2}$.  Both of our Ulrich sheaves supported on
$PX^s_{4,2}$ pull back to Ulrich sheaves supported on $\mathcal{E}_{\C}$, and their minimal free resolutions 
pull back to minimal free resolutions:

\[ s^*F \,\, \xleftarrow{\,\,\,\, s^*\alpha \,\,\,\,} \,\, s^*G \,\, \xleftarrow{\,\,\,\, s^*\phi \,\,\,\,} \,\, s^*G^* \,\, \xleftarrow{\,\,\,\, s^*\alpha^t \,\,\,\,} \,\, s^*F^*. \]

\noindent Here we verified in \texttt{Macaulay2} that $s^*$ quotients by a linear form that is a nonzero 
divisor for the two Ulrich modules.  So, to get the Chow form $\textup{Ch}(\mathcal{E}_{\C})$ from 
Propositions \ref{pro:UlrichFirst} and \ref{pro:UlrichSecond}, we took matrices $D_{1}$ and $D_{2}$ symmetrized from above, and applied $s^*$.
That amounts to substituting $x_{ij} = s(M)_{ij}$, where $s(M)$ is from \S \ref{subsec:coord}.  We then multiplied $D_{1}D_{2}D_{1}^T$, which 
is a product of a $20\times60$, a $60\times60$ and a $60\times20$ matrix, over the 
exterior algebra.

The two resulting explicit $20 \times 20$ skew-symmetric matrices are available 
at the paper's webpage.  
Their Pfaffians equal the 
Chow form of $\mathcal{E}_{\C}$, which is an element in the
homogeneous coordinate of $\textup{Gr}(\PP^2, \PP^8)$.  We denote that ring
as the polynomial ring in 84 (dual) Pl\"ucker variables 
$\Q[q_{\{11,12,13\}}, \ldots, q_{\{31,32,33\}}]$ 
modulo 1050 Pl\"ucker quadrics.  Here $\textup{Ch}(\mathcal{E}_{\C})$
lives in the 9,386,849,472-dimensional subspace of degree 10 elements.

Both matrices are excellent
representations of $\textup{Ch}(\mathcal{E}_{\C})$.  
Their entries are linear forms in $q_{\{11,12,13\}}, \ldots, q_{\{31,32,33\}}$ with relatively prime integer coefficients
less than 216 in absolute value.  In the first matrix, 96 off-diagonal entries 
vanish, and no entries have full support.

Bringing this back to computer vision, we can now prove our main result stated in \S \ref{sec:intro}:

\begin{proof}[Proof of Theorem \textup{\ref{mainThm}}]
Given $\{(x^{(i)},y^{(i)})\}$.  Let us first assume that we have a solution $A, B, \widetilde{X^{(1)}}, \ldots, \widetilde{X^{(6)}}$ to the system (\ref{3d}).
Note that the group:
$$G := \{ g \in \textup{GL}(4, \C) \, | \, (g_{ij})_{1 \leq i,j \leq 3} \in \textup{SO}(3, \C) \textup{ and } g_{41} = g_{42} = g_{43} = 0 \} $$
equals the stabilizer of the set of calibrated camera matrices inside $\C^{3 \times 4}$, with respect to right multiplication.
We~now~make~two~simplifying~assumptions~about~our~solution~to~$(\ref{3d})$.

\begin{itemize}

\item Without loss of generality, $A = [\, \textup{id}_{3 \times 3} \, | \, 0 \,]$.  For otherwise, select $g \in G$ so that
$Ag = [\, \textup{id}_{3 \times 3} \, | \, 0 \,]$, and then $Ag, Bg, g^{-1}\widetilde{X^{(1)}}, \ldots, g^{-1}\widetilde{X^{(6)}}$ is also a solution to $(\ref{3d})$.

\item Denoting $B = [\, R \, | \, t \,]$ for $R \in \textup{SO}(3, \C)$ and $t \in \C^{3}$, then without loss of generality, $t \neq 0$.  For otherwise, we may zero out the last coordinate of each $\widetilde{X^{(i)}}$ and replace
$B$ by $[\, R \, | \, t' \,]$ for any $t' \in \C^{3}$, and then we still have a solution to the system $(\ref{3d})$.  

\end{itemize}

Denote $[\, t \,]_{\times} := \begin{pmatrix} 0 & t_{3} & -t_{2} \\ -t_{3} & 0 & t_{1} \\ t_{2} & -t_{1} & 0 \end{pmatrix}$. Set $M = [ \, t \,]_{\times}R$.  Then $M \in \mathcal{E}_{\C}$.  
The following computation gives the basic link with $\textup{Ch}(\mathcal{E}_{\C})$:

\begin{align*}
 \begin{pmatrix} y^{(i)}_{1} & y^{(i)}_{2} & 1 \end{pmatrix} \,\, M \,\, \begin{pmatrix} x^{(i)}_{1} \\[3pt] x^{(i)}_{2} \\[3pt] 1 \end{pmatrix} \,\! &\equiv {(B \widetilde{X^{(i)}})}^{T}M \, (A\widetilde{X^{(i)}}) \\
 & = \widetilde{X^{(i)}}^{T} \Big{(} [ \, R \, | \, t \,]^{T}  \, [\, t \, ]_{\times} \, R \,\, [\, \textup{id}_{3 \times 3} \, | \, 0 \, ] \Big{)} \, \widetilde{X^{(i)}} \\
 & = \widetilde{X^{(i)}}^{T} \Big{(} [ \, R \, | \, 0 \,]^{T}  \, [\, t \, ]_{\times} \, [R \, | \, 0 \, ] \Big{)} \, \widetilde{X^{(i)}} \\
 & = 0.
\end{align*}

Here the second-to-last equality is because $t^{T} \, [\, t \, ]_{\times} = 0$, and the last equality is because the matrix in parentheses is skew-symmetric.
In particular, this calculation shows that $M \in \mathcal{E}_{\C}$ satisfies six linear constraints.  Explicitly, these are:

\[\begin{pmatrix}
\\[-10pt]
y^{(1)}_{1}x^{(1)}_{1} & y^{(1)}_{1}x^{(1)}_{2} & y^{(1)}_{1} &
y^{(1)}_{2}x^{(1)}_{1} & y^{(1)}_{2}x^{(1)}_{2} & y^{(1)}_{2} & 
x^{(1)}_{1} & x^{(1)}_{2} & 1\\[8pt]
y^{(2)}_{1}x^{(2)}_{1} & y^{(2)}_{1}x^{(2)}_{2} & y^{(2)}_{1} &
y^{(2)}_{2}x^{(2)}_{1} & y^{(2)}_{2}x^{(2)}_{2} & y^{(2)}_{2} & 
x^{(2)}_{1} & x^{(2)}_{2} & 1\\[8pt]
y^{(3)}_{1}x^{(3)}_{1} & y^{(3)}_{1}x^{(3)}_{2} & y^{(3)}_{1} &
y^{(3)}_{2}x^{(3)}_{1} & y^{(3)}_{2}x^{(3)}_{2} & y^{(3)}_{2} & 
x^{(3)}_{1} & x^{(3)}_{2} & 1\\[8pt]
y^{(4)}_{1}x^{(4)}_{1} & y^{(4)}_{1}x^{(4)}_{2} & y^{(4)}_{1} &
y^{(4)}_{2}x^{(4)}_{1} & y^{(4)}_{2}x^{(4)}_{2} & y^{(4)}_{2} & 
x^{(4)}_{1} & x^{(4)}_{2} & 1\\[8pt]
y^{(5)}_{1}x^{(5)}_{1} & y^{(5)}_{1}x^{(5)}_{2} & y^{(5)}_{1} &
y^{(5)}_{2}x^{(5)}_{1} & y^{(5)}_{2}x^{(5)}_{2} & y^{(5)}_{2} & 
x^{(5)}_{1} & x^{(5)}_{2} & 1\\[8pt]
y^{(6)}_{1}x^{(6)}_{1} & y^{(6)}_{1}x^{(6)}_{2} & y^{(6)}_{1} &
y^{(6)}_{2}x^{(6)}_{1} & y^{(6)}_{2}x^{(6)}_{2} & y^{(6)}_{2} & 
x^{(6)}_{1} & x^{(6)}_{2} & 1 \\[5pt]
\end{pmatrix}
\, \begin{pmatrix}
m_{11} \\ m_{12} \\ m_{13} \\ 
m_{21} \\ m_{22} \\ m_{23} \\
m_{31} \\ m_{32} \\ m_{33} \\[1pt]
\end{pmatrix}
\,\,\, = \,\,\, 0.
\]

Let the above $6 \times 9$ matrix be denoted $Z$. 
We consider two cases.

\begin{itemize}

\item \underline{Case $1$: \textit{$Z$ is full rank.}}  Then $\textup{ker}(Z)$ determines a $\PP^{2}$ in $\PP^{8}$.
This $\PP^{2}$ meets $\mathcal{E}_{C}$, 
namely at $M$.  So, $\text{Ch}(\mathcal{E}_{\C})$ 
evaluates to 0 there.  By \cite[p.94]{GKZ}, we can compute
the Pl\"ucker coordinates of this projective plane from the maximal minors of $Z$.

\item \underline{Case $2$: \textit{$Z$ is not full rank.}}  Then all maximal minors of $Z$ are 0.

\end{itemize}

Thus, to get $\mathcal{M}(x^{(i)},y^{(i)})$ as in Theorem \ref{mainThm}, 
we take either of the $20 \times 20$ skew-symmetric 
matrix formulas for $\textup{Ch}(\mathcal{E}_{\C})$ described above, and we replace each $q_{ijk}$ by 
the determinant of $Z$ with columns $i, j$ and $k$ removed.
In Case 1, this $\mathcal{M}(x^{(i)},y^{(i)})$ drops rank, by the definition of
Chow forms.  In Case 2, this $\mathcal{M}(x^{(i)},y^{(i)})$ evaluates to the zero matrix.
We have proven that this $\mathcal{M}(x^{(i)},y^{(i)})$ satisfies the first property 
stated in Theorem \ref{mainThm}.

We now prove that this $\mathcal{M}(x^{(i)},y^{(i)})$ satisfies the converse property in Theorem \ref{mainThm}.
Factor $M = U\,\text{diag}(1, 1, 0)\,V^{T}$ 
with $U, V \in \text{SO}(3, \C)$.  This is possible for a Zariski open 
subset of $M \in \mathcal{E}_{\C}$. For the dense subset in Theorem \ref{mainThm}, we take those $\{ (x^{(i)}, y^{(i)}) \}$ for which
there is $M$ in the above Zariski open subset such that 
$ \widetilde{y^{(i)}}^{T}M\,\widetilde{x^{(i)}} = 0 $. This is a dense open subset in all pairs $\{ (x^{(i)}, y^{(i)}) \}$
such that $\mathcal{M}(x^{(i)},y^{(i)})$ is rank deficient.
 Denote 
$W = \begin{pmatrix} 0 & -1 & 0 \\ 1 & 0 & 0 \\ 0 & 0 & 1 \end{pmatrix}$.
Now set 
$A = \begin{pmatrix} I \!\!\!\!& | &\!\!\!\! 0 \end{pmatrix}$ and 
$B =  \begin{pmatrix} \, UWV^{T} \!\!& | &\!\! U{\begin{pmatrix} 0 & 0 & 1\end{pmatrix}}^{T}  \end{pmatrix}$.
Now $\widetilde{X^{(i)}}$ are uniquely determined (see \cite[9.6.2]{HZ}).
\end{proof}

We illustrate the main theorem with two examples.  
Note that since the first example is a `positive', 
it is a strong (and reassuring)
check of correctness for our formulas.

\begin{example}\label{ex:positive}
Consider the image data of 6 point correspondences 
$\{(x^{(i)},y^{(i)}) \in \RR^2\times\RR^2 \, | \, i=1,\ldots, m\}$
given by the corresponding rows of the
two matrices:
\[
[\,x^{(i)}\,] \,=\, \begin{pmatrix}
0 & 0 \\ 1 & -1 \\[1pt] 0 & -\frac{1}{2} \\[2pt] -3 & 0 \\[2pt] \frac{3}{2} & -\frac{5}{2} \\[3pt] 1 & \frac{1}{7}
\end{pmatrix}
\hspace{2cm}
[\,y^{(i)}\,] \,=\, \begin{pmatrix}
\frac{8}{11} & \frac{16}{11} \\[4pt] \frac{7}{22} & \frac{5}{22} \\[4pt] \frac{8}{29} & \frac{34}{29} \\[4pt] \frac{17}{20} & -1 \\[4pt] \frac{1}{7} & \frac{1}{7} \\[4pt] \frac{9}{4} & \frac{3}{4}
\end{pmatrix}.
\]
In this example, they do come from world points $X^{(i)} \in \RR^3$ and calibrated cameras $A, B$:
\[
\big{[}\,X^{(i)}\,\big{]} \,=\, \begin{pmatrix}
0 & 0 & 2 \\ 1 & -1 & 1 \\ 0 & -2 & 4 \\ 3 & 0 & -1 \\ 3 & -5 & 2 \\ 7 & 1 & 7
\end{pmatrix},
\hspace{1cm}
A \,=\, \begin{pmatrix}
1 & 0 & 0 & 0 \\ 0 & 1 & 0 & 0 \\ 0 & 0 & 1 & 0
\end{pmatrix},
\hspace{1cm}
B \,=\, \begin{pmatrix}
\frac{7}{9} & \frac{4}{9} & \frac{4}{9} & 0 \\[4pt] -\frac{4}{9} & -\frac{1}{9} & \frac{8}{9} & 0 \\[4pt] \frac{4}{9} & -\frac{8}{9} & \frac{1}{9} & 1
\end{pmatrix}.
\]
To detect this, we form the $6 \times 9$ matrix $Z$ from the proof of Theorem \textup{\ref{mainThm}}:
\[
Z \,=\, \begin{pmatrix}
0 & 0 & \frac{8}{11} & 0 & 0 & \frac{16}{11} & 0 & 0 & 1 \\[4pt]  
\frac{7}{22} & -\frac{7}{22} & \frac{7}{22} & \frac{5}{22} & -\frac{5}{22} & \frac{5}{22} & 1 & -1 & 1 \\[4pt]  
0 & -\frac{4}{29} & \frac{8}{29} & 0 & -\frac{17}{29} & \frac{34}{29} & 0 & -\frac{1}{2} & 1 \\[4pt]  
-\frac{51}{20} & 0 & \frac{17}{20} & 3 & 0 & -1 & -3 & 0 & 1 \\[4pt]  
\frac{3}{14} & -\frac{5}{14} & \frac{1}{7} & \frac{3}{14} & -\frac{5}{14} & \frac{1}{7} & \frac{3}{2} & -\frac{5}{2} & 1 \\[4pt]  
\frac{9}{4} & \frac{9}{28} & \frac{9}{4} & \frac{3}{4} & \frac{3}{28} & \frac{3}{4} & 1 & \frac{1}{7} & 1
\end{pmatrix}.
\]
We substitute the maximal minors of $Z$ into the matrices computing 
$\textup{Ch}(\mathcal{E}_{\C})$ in \textup{\texttt{Macaulay2}}.  The determinant command 
then outputs 0.  This computation recovers the fact that
the point correspondences are images 
of 6 world points under a pair of calibrated cameras.

\end{example}

\begin{example}\label{ex:negative}
Random data 
$\{(x^{(i)},y^{(i)}) \in \RR^2\times\RR^2 \, | \, i=1,\ldots, 6\}$
is expected to land outside the Chow divisor of $\mathcal{E}_{\C}$.
We made an instance using the \textup{\texttt{random(QQ)}} 
command in \textup{\texttt{Macaulay2}} for each coordinate of 
image point.  The coordinates ranged 
from $\frac{1}{8}$ to 5 in absolute value.
We carried out the substitution from Example \textup{\ref{ex:positive}},
and got two full-rank skew-symmetric matrices with Pfaffians 
$\approx 5.5 \times 10^{25}$ and $\approx 1.3 \times 10^{22}$, respectively.
These matrices certified that the system \textup{(\ref{3d})} 
admits no solutions
for that random input.

\end{example}

The following proposition is based on general properties of Chow forms, 
collectively known as the U-resultant method to
solve zero-dimensional polynomial systems.
In our situation, it gives a connection with the `five-point algorithm' 
for computing essential matrices.  The
proposition is computationally inefficient as-is for that purpose, but
see \cite{MKF} for a more efficient algorithm that would exploit 
our matrix formulas for $\text{Ch}(\mathcal{E}_{\C})$.  
Implementing
the algorithms in \cite{MKF} for our matrices is one avenue for future work.

\begin{prop}\label{recoverEssentials}
Given a generic 5-tuple 
$\{(x^{(i)},y^{(i)}) \in \RR^2\times\RR^2 \, | \, i=1,\ldots, 5\}$,
if we make the substitution from the proof of 
Theorem \textup{\ref{mainThm}}, then the Chow form $\textup{Ch}(\mathcal{E}_{\C})$
specializes to a polynomial in 
$\RR[x^{(6)}_{1}, x^{(6)}_{2}, y^{(6)}_{1}, y^{(6)}_{2}]$.
Over $\C$, this specialization completely splits as:
\[
\prod_{i=1}^{10} \small{\begin{pmatrix} y^{(6)}_{1} & y^{(6)}_{2} & 1 \end{pmatrix} \,\, M^{(i)} \,\, \begin{pmatrix} x^{(6)}_{1} \\[3pt] x^{(6)}_{2} \\[3pt] 1 \end{pmatrix}.}
\]
Here $M^{(1)}, \ldots, M^{(10)} \in \mathcal{E}_{\C}$ are the essential
matrices determined by the given five-tuple.
\end{prop}

\begin{proof}
By the proof of Theorem \ref{mainThm},
any zero of the above product is a zero of the 
specialization
of $\textup{Ch}(\mathcal{E}_{\C})$.
By Hilbert's Nullstellensatz, this
implies that the product divides 
the specialization.
But both polynomials
are inhomogeneous 
of degree 20, so they are $\equiv$.
\end{proof}

\subsection{Numerical experiments with noisy point correspondences}

In this final subsection, we discuss
how our Theorem \ref{mainThm} 
is actually resistant to a common complication
in concrete applications of algebra: noisy data.
Indeed, on real image data, correctly matched point pairs 
will only come to the computer vision practitioner with finite accuracy.
In other words, they differ from exact 
correspondences by some noise.  

\begin{question}
While in Theorem \textup{\ref{mainThm}} the matrix $\mathcal{M}(x,y)$
drops rank when there is an exact solution to \textup{(\ref{3d})},
how can we tell if there is an approximate solution?
\end{question}

The answer is to calculate the Singular Value Decomposition of the matrices
$\mathcal{M}(x,y)$ from Theorem \ref{mainThm}, when a noisy six-tuple of 
image point correspondences is plugged in.
Since Singular Value Decomposition is numerically stable \cite[\S 5.2]{Dem},
we expect approximately rank-deficient SVD's when there exists 
an approximate solution to ($\ref{3d}$).  To summarize,
since we have \textbf{matrix} formulas, we can look at \textbf{spectral gaps}
in the presence of noise.

We offer experimental evidence that this works.
For our experiments, we assumed uniform noise from
$\textup{unif}\, [-10^{-r}, \, 10^{-r}]$; 
this arises in image processing from
pixelation \cite[\S 4.5]{Bov}.
For each $r=1, \, 1.5, \, 2, \, \ldots, \,15$, we 
executed five hundred of the following trials: 

\begin{itemize}
\item \textit{Pseudo-randomly generate an exact six-tuple
of image point correspondences 
\[\{(x^{(i)},y^{(i)}) \in \Q^2\times\Q^2 \, | \, i=1,\ldots, 6\}\]
with coordinates of size $O(1)$.}
\item \textit{Corrupt each image coordinate in the six-tuple by 
adding an independent and identically distributed
sample from $\textup{unif}\, [-10^{-r}, \, 10^{-r}]$.}
\item \textit{Compute the SVD's of both $20 \times 20$ matrices $\mathcal{M}(x,y)$, derived from the
first and second Ulrich sheaf respectively, with the above noisy image coordinates plugged in.}
\end{itemize}

\begin{figure}
\centering
\begin{tikzpicture} \begin{axis} [
xlabel = {Accuracy of point correspondences},
ylabel = {Average signal from SVD},
grid = major,
legend style={at={(0.07,0.9)},anchor=north west,font=\tiny},
legend entries = {matrix from first sheaf, matrix from second sheaf},
]
\addplot coordinates {
    (1,1.16028287852412)
(1.5,1.3916948800606)
(2,1.82692894733035)
(2.5,2.21894855673212)
(3,2.72710939053883)
(3.5,3.22762523555921)
(4,3.69784088706456)
(4.5,4.29547440071266)
(5,4.73575492783225) 
(5.5,5.21591529521078)
(6,5.7183196871869)
(6.5,6.17193012807508)
(7,6.65064456129844)
(7.5,7.23623193146843)
(8,7.71318514498591)
(8.5,8.19410157974031)
(9,8.73029922757801)
(9.5,9.20085675008298)
(10,9.67129198485323)
(10.5,10.2286257334291)
(11,10.7767927717677)
(11.5,11.2497846723665)
(12,11.7195473226768)
(12.5,12.2295210295655)
(13,12.8087154722628)
(13.5,13.2412886250093)
(14,13.6174894124508)
(14.5,13.9045453015374)
(15,14.0279980412893)
};
\addplot coordinates {};
\addplot coordinates {};
\addplot coordinates {};
\addplot coordinates {};
\addplot coordinates {};
\addplot coordinates {};
\addplot coordinates {};
\addplot coordinates {};
\addplot coordinates {
(1,.836434244010548)
(1.5,1.11094064328223)
(2,1.57122967278171)
(2.5,1.97200044908366)
(3,2.4833637394284)
(3.5,2.99507985859513)
(4,3.48566006144779)
(4.5,4.06883705473918)
(5,4.55564229441522)
(5.5,4.99792616504173)
(6,5.46238193854507)
(6.5,5.96490204951661)
(7,6.44087593126678)
(7.5,6.97719036750883)
(8,7.47416539835522)
(8.5,7.94542041770554)
(9,8.50609424431121)
(9.5,8.91615424828249)
(10,9.49354884739002)
(10.5,9.99418295365698)
(11,10.5285534281416)
(11.5,10.99816462228)
(12,11.4507989884858)
(12.5,11.9895041433225)
(13,12.5624250663382)
(13.5,12.9737765718665)
(14,13.3292359125474)
(14.5,13.6558268669346)
(15,13.7760581282421)
};
\end{axis}
\end{tikzpicture}
\captionsetup{labelformat=empty}
\caption{\hspace{1.5cm} FIGURE. Both matrices satisfying Theorem \ref{mainThm}}  
\vspace{-0.35cm}
\caption{\hspace{1.5cm} detect approximately consistent point pairs.} \label{fig:exps}
\end{figure}

These experiments were performed
in \texttt{Macaulay2} using double precision
for all floating-point arithmetic.
Since it is a little subtle, we elaborate on our algorithm to
pseudo-randomly generate
exact correspondences in the first 
bullet.  It breaks into three steps:

\begin{enumerate}
\item Generate calibrated 
cameras $A, B \in \Q^{3 \times 4}$. 
To do this, we sample twice from the Haar measure on 
$\text{SO}(3,\RR)$ and sample twice from the 
uniform measure 
on the radius 2 ball centered at the origin 
in $\RR^{3}$.  Then we concatenate nearby 
points in
$\text{SO}(3,\Q)$ and $\Q^{3}$ to obtain $A$ and $B$.
To find the nearby rotations, 
we pullback under $\RR^{3} \mto{} S^{3} \backslash \{N\} \mto{} \text{SO}(3, \RR)$, 
we take nearby points in $\Q^3$, and then we pushforward.

\item Generate world points 
$X^{(i)} \in \Q^{3} \,\, (i=1, \ldots, 6)$.
To do this, we sample six times from the uniform measure
on the radius 6 ball centered at the origin in 
$\RR^{3}$ (a choice
fitting with some real-world data)
and then we replace those by nearby
points in $\Q^{3}$.

\item  Set
$\widetilde{x^{(i)}} \equiv A\widetilde{X^{(i)}}$ and
$\widetilde{\, y^{(i)}} \equiv B\widetilde{X^{(i)}}$.
\end{enumerate}

The most striking takeaway of our experiments
is stated in the following result 
concerning the bottom spectral gaps we observed.
Bear in mind that since
$\mathcal{M}(x,y)$ is skew-symmetric, 
its singular values occur with multiplicity 
two, so $\sigma_{19}(\mathcal{M}(x,y)) = \sigma_{20}(\mathcal{M}(x,y))$.

\begin{empirical}\label{empirical}
In the experiments described above, we observed for both matrices:
$$\frac {\sigma_{18}(\mathcal{M}(x,y))}{\sigma_{20}(\mathcal{M}(x,y))} = O(10^{r}).$$
Here $\mathcal{M}(x,y)$ has r-noisy image coordinates, and
$\sigma_{i}$ denotes the $i^{\textup{th}}$ largest singular value.
\end{empirical}

\noindent  The figure above plots
$\log_{10} \begin{pmatrix} \frac{\sigma_{18}(\mathcal{M}(x,y))}{\sigma_{20}(\mathcal{M}(x,y))} \end{pmatrix}$ averaged over the five hundred trials against 
$r$.

\bigskip \medskip 

\footnotesize 

\noindent \textbf{Authors' addresses:}

\noindent
Gunnar Fl{\o}ystad, Universitetet i Bergen, Norway,
\texttt{gunnar@mi.uib.no}

\noindent Joe Kileel, University of California, Berkeley, USA,
\texttt{jkileel@math.berkeley.edu}

\noindent Giorgio Ottaviani, Universit{\`a} di Firenze, Italy,
\texttt{ottavian@math.unifi.it}

\end{document}